\documentclass{amsart}
\usepackage[utf8]{inputenc}
\usepackage[english]{babel}

\usepackage{amsfonts}
\usepackage{amsmath}
\usepackage{amssymb}
\usepackage{amsthm}
\usepackage[margin=1in]{geometry}
\usepackage{tikz-cd}
\usepackage{soul}
\usepackage{scalerel}

\usepackage{xcolor}

\usepackage{hyperref}

\theoremstyle{definition}
\newtheorem{definition}{Definition}[section]

\theoremstyle{definition}
\newtheorem{prop}{Proposition}[section]

\theoremstyle{definition}
\newtheorem{conj}{Conjecture}[section]

\theoremstyle{definition}
\newtheorem{thm}{Theorem}

\newtheorem{lem}{Lemma}
\newtheorem*{rmk}{Remark}

\newtheorem*{ex}{Example}

\title{Hecke Action on Tamely Ramified Eisenstein Series over $\mathbb{P}^1$}
\author{Tahsin Saffat} 
\date{2023}

\begin{document}

\maketitle

\begin{abstract} We study the space of automorphic functions for the rational function field $\mathbb{F}_q(t)$ tamely ramified at three places. Eisenstein series are functions induced from the maximal torus. The space of Eisenstein series generates a trimodule for the affine Hecke algebra. We conjecture a generators and relations description of this module and prove the conjecture when $G=\mathrm{PGL}(2)$ and $\mathrm{SL}(3)$.
\end{abstract}

\tableofcontents

\section{Introduction}

\subsection{Problem Setup}

Let $\mathbb{P}^1$ denote the projective line over $\mathbb{F}_q$.  We denote by $\mathbb{F}$, $\mathbb{A}$, and $\mathbb{O}$, its function field, ring of adeles, and the subring of integral adeles, respectively. Let $S= \{0,1,\infty\} \subset \mathbb{P}^1(\mathbb{F}_q)$ and fix $G$ a reductive group over $\mathbb{F}_q$, $T$ a split maximal torus split, and $B$ a Borel subgroup containing $T$.

Define $K_S:= \prod_{x \in |\mathbb{P}^1|}K_x$, where $K_x=I\subset G(\mathbb{F}_x)$ is the Iwahori subgroup for $x \in S$ and $K_x=G(\mathbb{F}_x)$ otherwise. Consider the vector space of compactly supported, complex valued automorphic functions $$C_{Aut}:=C\left[ K_S \backslash G(\mathbb{A})  / G(\mathbb{F}) \right].$$

\noindent This vector space has an action of the affine Hecke algebra at $x \in S$ and the spherical Hecke algebra at $x \in |\mathbb{P}^1| \setminus S $. Given a function $\phi:\Lambda \rightarrow \mathbb{C}$
, the Eisenstein series $\mathrm{Eis}_{\phi}$ is a ($G(\mathbb{O})$ invariant) function on $G(\mathbb{A})/G(\mathbb{F})$ defined by

$$\mathrm{Eis}_{\phi}(g):=\sum_{\gamma \in G(\mathbb{F})/B(\mathbb{F})} \phi(g \gamma)$$

For $\lambda \in \Lambda$, define $\mathrm{Eis}_{\lambda}=\mathrm{Eis}_{\underline{1}_{\lambda}}$, where $\underline{1}_{\lambda}$ is the function taking value $1$ on $\lambda$ and $0$ elsewhere.

\noindent We are interested in the space $C_{Eis}$, which is the closure under all Hecke operators of

$$\mathrm{span}\{\mathrm{Eis}_{\lambda}\}\subset C\left[ K_S \backslash G(\mathbb{A})  / G(\mathbb{F}) \right].$$

\noindent We explicitly determine the action of all Hecke operators on $C_{Eis}$. In order to state the main result, we first reformulate the problem.

\subsection{Geometric Formulation}
Let $\mathrm{Bun}_G(\mathbb{P}^1,S)$ denote the moduli stack over $\mathbb{F}_q$ of $G$ bundles on $\mathbb{P}^1$ with $B$ reduction along $S$. For example, for $G=GL_n$, it classifies the data $(\mathcal{E},\{F_s\}_{s \in S})$ where $\mathcal{E}$ is a rank $n$ vector bundle on $\mathbb{P}^1$ and $F_s$ is a flag in $\mathcal{E}|_s$. Then, $K_S \backslash G(\mathbb{A})  / G(\mathbb{F})$ is identified with the $\mathbb{F}_q$ points of $\mathrm{Bun}_G(\mathbb{P}^1,S)$. There is an induction diagram

$$\mathrm{Bun}_T(\mathbb{P}^1) \xleftarrow{p} \mathrm{Bun}_B(\mathbb{P}^1) \xrightarrow{q} \mathrm{Bun}_G(\mathbb{P}^1,S),$$

\noindent Moreover, $\mathrm{Bun}_T(\mathbb{P}^1)$ is identified with $\Lambda \otimes \mathrm{Pic}(\mathbb{P}^1)$ and each $\lambda \in \Lambda$ corresponds to a component of $\mathrm{Bun}_T(\mathbb{P}^1)$. The Eisenstein series corresponding to $\lambda \in \Lambda$  is constructed as $\mathrm{Eis}_{\lambda}=q_!p^*(\underline{1}_{\lambda})$, where $\underline{1}_{\lambda}$ is the function taking value $1$ along the component corresponding to $\lambda$ and $0$ elsewhere. The pushforward is integration along fibers relative to the motivic (or weighted counting) measure.

The action of spherical Hecke operators on $C_{Eis}$ at any $x \in |\mathbb{P}^1| \setminus S $ can be expressed in terms of the Hecke operators at $s \in S$ through the central homomorphism $\mathcal{H}^{sph} \rightarrow \mathcal{H}$. 
We denote by $\mathcal{H}^{\otimes S}$ the algebra generated by Hecke operators at all $s \in S$. For $s \in S$, picking a uniformizer in the local ring $\mathbb{F}_s$ identifies the algebra of Hecke operators at $s$, $\mathcal{H}^s$, with the algebra of compactly supported functions on the $\mathbb{F}_q$ points of $I \backslash G((t)) / I$.

For $w \in W^{aff}$, define $T_w$ as the corresponding function on $I \backslash G((t)) / I$.  There is an injective homomorphism $\mathbb{C}[\Lambda] \rightarrow \mathcal{H}$ defined by $\lambda \mapsto T_{\lambda}$ for $\lambda$ antidominant. Let $J_{\lambda}$ denote the image of $\lambda$ under this homomorphism. For an operator $A \in C[I \backslash G((t))/I]$ and $s \in S$, let $A^s$ be its image under the identification

$$C[I \backslash G((t))/I] \cong \mathcal{H}^s$$

\subsection{Main Result}

\noindent Let us introduce a mild technical restriction on $G$. Let $\Lambda^{\vee}:= \mathrm{Hom}(T,\mathbb{G}_m)$ denote the lattice of weights of $T$. Assume that for all roots $\check{\alpha} \in \Lambda^{\vee}$, the map $\Lambda \rightarrow \mathbb{Z}$ given by $\lambda \mapsto \langle \check{\alpha},\lambda \rangle$ is surjective. For example, $\mathrm{PGL}_2$ and $\mathrm{SL}_3$ satisfy this condition, but $\mathrm{SL}_2$ does not. The adjoint form of a group will always satisfy this condition.  We conjecture the following.

\begin{conj}\label{mainconj}
$C_{Eis}$ is the $\mathcal{H}^{\otimes S}$ module generated by $\mathrm{Eis}_0$ with the following relations

\begin{enumerate}
   
    \item (Translation Relation)
    For any $\lambda \in \Lambda$
    $$J_{\lambda}^0\mathrm{Eis}_0=J_{\lambda}^1\mathrm{Eis}_0 =J_{\lambda}^{\infty}\mathrm{Eis}_0$$
   
    \item (Reflection Relation)
    For any simple reflection, $s_{\alpha} \in W$
    $$(1+T_{s_\alpha}^0)(1+T_{s_\alpha}^1)\mathrm{Eis}_0 =(1+T_{s_\alpha}^0)(1+T_{s_\alpha}^{\infty})\mathrm{Eis}_0 =(1+T_{s_\alpha}^1)(1+T_{s_\alpha}^{\infty})\mathrm{Eis}_0 $$
    
\end{enumerate}
\end{conj}

There is a natural generalization of this, Conjecture \ref{genconj}, to arbitrary tame ramification $S \subset \mathbb{P}^1(\mathbb{F}_q)$. When $S$ consists of one or two points, the conjecture follows from the Radon transform which identifies $C_{Eis}$ with the regular bimodule for the Hecke algebra. 
In this paper, we prove Conjecture \ref{mainconj} when $G=\mathrm{PGL}(2)$ or $\mathrm{SL}(3)$ (Theorems \ref{pgl2thm} and \ref{sl3thm}) as well as in the following generic sense. Let $\widetilde{C}$ be the quotient of $\mathcal{H}^{\otimes S}$ by the left ideal generated by the Translation and Reflection relations. 

\begin{thm}\label{mainthm}
There is a surjective map $\widetilde{C} \rightarrow C_{Eis}$ of $\mathcal{H}^{\otimes S}$ modules such that rationalizing the action of translation operators at $0$ yields an isomorphism

$$\mathrm{Frac}(\mathbb{C}[\Lambda]) \otimes_{\mathbb{C}[\Lambda]}\widetilde{C} \xrightarrow{\cong} \mathrm{Frac}(\mathbb{C}[\Lambda]) \otimes_{\mathbb{C}[\Lambda]} C_{Eis} .$$

\noindent The proof of Theorem \ref{mainthm} is given in Section \ref{proofmainthm}.
    
\end{thm}

\subsection{Acknowledgements}
I thank David Nadler for suggesting this problem and for providing extremely generous support. I also thank Zhiwei Yun for helpful discussions and for suggesting the connection to the functional equation for Eistenstein series.
This work was partialy supported by NSF grant DMS-1646385.
\section{Background on Group Theory, Hecke Operators, and Eisenstein Series}

\subsection{Operations with Functions}

In this paper we will compute pushforwards and pullbacks of functions on rational points of stacks. The set of rational points of an Artin stack, $X$, over $\mathbb{F}_q$, has a natural measure, given by 
 $\mu(\mathcal{E})=|\mathrm{Aut}(\mathcal{E})|^{-1}$.

This endows the space, $C(X)$, of (compactly supported) functions on the rational points with an inner product. Given a map $f:X \rightarrow Y$, there is a pullback, $f^*:C(Y) \rightarrow C(X)$, given by $f^*F(x)=F(f(x))$. Pushforward $f_!:C(X) \rightarrow C(Y)$ is the adjoint of pullback with respect to the inner product

$$\int_{X(\mathbb{F}_q)} d\mu_X F_1(x) f^*F_2(x)=\int_{Y(\mathbb{F}_q)}d \mu_Y f_!F_1(y)F_1(y)$$

\begin{ex}
  A group homomorphism $H \hookrightarrow G$ induces  a map $f:\mathrm{pt}/H \rightarrow \mathrm{pt}/G$. Identifying, $C(\mathrm{pt}/H) \cong \mathbb{C} \cong C(\mathrm{pt}/G)$, $f^*=1 \in  \mathrm{End}(\mathbb{C})$ and $f_!=[G(\mathbb{F}_q):H(\mathbb{F}_q)] \in \mathrm{End}(\mathbb{C})$.
\end{ex}

\begin{ex}[Base Change] 
Given a Cartesian diagram, $g^*v_!=f_!u^*$.

\begin{center}

\begin{tikzcd}
    X \arrow[r,"f"] \arrow[d,"u"] & Y \arrow[d,"g"] \\
    Z \arrow[r,"v"] & W
\end{tikzcd}

\end{center}

\noindent A special case of the base change formula is that given $f:X \rightarrow Y$, 
$$f_!F_1 \otimes F_1=f_!(F_1 \otimes f^*F_2$$

\noindent where $\otimes$ denotes pointwise multiplication of functions.

\end{ex}

\begin{ex}[Finite Hecke Algebra]
For $G$ a spit reductive group over $\mathbb{F}_q$ with Borel subgroup $B \subset G$, the convolution product on $C(B \backslash G/B)$ is realized through the following diagram.

\begin{centering}

\begin{tikzcd}
 B \backslash G /B& B \backslash G \times_{B} G/B \arrow[r,"\mathrm{prod}"] \arrow[d,"\pi_2"] \arrow[l,"\pi_1"] &  B \backslash G /B \\
&  B \backslash G /B
& 
\end{tikzcd}

\end{centering}

\noindent For functions, $F_1,F_2 \in C(B \backslash G/B)$, 

$$F_1 \cdot F_2=\mathrm{prod}_{!}(\pi_1^*F_1 \otimes \pi_2^*F_2)$$

\end{ex}

\subsection{Group Theory}
Fix a reductive group, $G$ over $\mathbb{F}_q$ and assume $G$ has a split torus.

Fix a choice of Borel $B \subset G$ and let $N \subset B$ be the unipotent radical and $T=B/N$ the universal Cartan. $\Lambda=\mathrm{Hom}(\mathbb{G}_m,T)$ is the coweight lattice with $R_{+} \subset \Lambda$ the positive coroots, and $\Lambda^{\vee}=\mathrm{Hom}(T,\mathbb{G}_m)$ the weight lattice with $R_{+}^{\vee} \subset \Lambda^{\vee}$ the positive roots. $\rho \in \Lambda$ is half the sum of the elements of $R_+$

$\Lambda_+ \subset \Lambda$ is the dominant cone. Let $W$ denote the Weyl group and $W^{aff} \cong W \ltimes \Lambda$ the (extended) affine Weyl group. Let $\mathcal{B} \cong G/B$ denote the flag variety. $G((t)):=\mathrm{Map}(\mathrm{Spec}\mathbb{F}_q((t)),G)$ is the loop group and $G[[t]]:=\mathrm{Map}(\mathrm{Spec}\mathbb{F}_q[[t]],G)$ is the arc group. $I \subset G[[t]]$ is the Iwahori subgroup, defined as the preimage of $B$ under evaluation at zero $\mathrm{ev}_0:G[[t]] \rightarrow G$.

\subsection{Affine Hecke Algebra}

The affine Hecke algebra is the following algebra of compactly supported functions under convolution. Normalize the Haar measure on $G((t))$ so that $I$ has unit measure.

$$\mathcal{H}:=C\left[ I\backslash G((t))/I \right].$$

\noindent The points of $I \backslash G((t)) / I$ are indexed by elements of $W^{aff}$. Given $w \in W^{aff}$, let $T_w$ denote the corresponding  element of $\mathcal{H}$.

\noindent $I \backslash G((t)) / I$ classifies pairs of bundles on the formal disk $\mathrm{Spec}(\mathbb{F}_q[[t]])$ along with an isomorphism of their restrictions away from zero. 

\begin{definition}[Translation Operator]
For $\lambda \in \Lambda$, define the \textit{translation operator} $J_{\lambda}=(T_{-\lambda_1})^{-1}T_{-\lambda_2}$, where $\lambda=\lambda_1-\lambda_2$, with $\lambda_i \in \Lambda_{+}$.
\end{definition}

\begin{rmk}
The definition does not depend on the choice $\lambda_1,\lambda_2$.
\end{rmk}

\begin{thm}[Bernstein's Relations]\label{classicthm}

    The operators $T_w$ $w \in W$ form a basis for the subalgebra $\mathcal{H}^{fin}$, called the finite Hecke algebra. The relations are as follows:

\begin{itemize}
\item $T_{w_1}T_{w_2}=T_{w_1w_2}$ if $ \ell(w_1w_2)=\ell(w_1)+\ell(w_2)$
\item $T_{s_{\alpha}}^2=(q-1)T_{s_{\alpha}}+q$ if $s_{\alpha} \in W$ is $\mathrm{simple}$
\end{itemize}

\end{thm}
The operators $T_{s_{\alpha}}$ and $J_{\lambda}$ for $s_{\alpha} \in W$ simple and $\lambda \in \Lambda$ satisfy

$$ J_{\lambda}T_{s_{\alpha}}=q^{-\check{\alpha}(\lambda)}T_{s_{\alpha}}J_{s_{\alpha}(\lambda)}+(q-1)\frac{J_{\lambda}-q^{-\check{\alpha}(\lambda)}J_{s_{\alpha}(\lambda)}}{1-qJ_{\alpha}}$$

\begin{proof}
See Proposition 3.6 of \cite{Lusztig} for the original proof by Lusztig based on unpublished work of Bernstein. \cite{hkp}, \cite{goertz}, and \cite{ram} were also helpful references for the author.
    
\end{proof}

%\begin{rmk}

It follows that 
 $\lambda \mapsto J_{\lambda}$ is an injective homomorphism $\mathbb{C}[\Lambda] \rightarrow H^{aff}$. Its image, $A$, a maximal commutative subalgebra.

The geometric basis elements, $\{T_w\}_{w \in W^{aff}}$, are partially ordered by the length function on $W^{aff}$.
Both the sets $\{J_{\lambda}T_w\}_{w \in W, \lambda \in \Lambda}$ and $\{T_w J_{\lambda}\}_{w \in W, \lambda \in \Lambda}$ are upper triangular with respect to the geometric basis, by Theorem \ref{classicthm}. In particular, they are bases.

\subsection{Hecke Operators}
For $s \in S$ the Hecke operators $\mathcal{H}^s$ are constructed as follows. We construction a left action of $\mathcal{H}$ on $C_{Aut}$. Picking a uniformizer in the completed local ring at $s$ defines a map $\mathrm{Spec}(\mathbb{F}_q[[t]]) \rightarrow \mathbb{P}^1$ sending the closed point to $s$. consider the following diagram

\begin{centering}

\begin{tikzcd}
 \mathrm{Bun}_G(\mathbb{P}^1,S)& \mathrm{Corr}^s  \arrow[r,"\pi_2"] \arrow[d,"\mathrm{res}"] \arrow[l,"\pi_1"] & \mathrm{Bun}_G(\mathbb{P}^1,S) \\
& I \backslash G((t)) / I
& 
\end{tikzcd}

\end{centering}

\noindent $\mathrm{Corr}^s$ classifies data the data of a triple $(\mathcal{E}_1,\mathcal{E}_2,T)$ where $\mathcal{E}_1,\mathcal{E}_2$ are parabolic $G$-bundles  on $\mathbb{P}^1$ and $T$ is an isomorphism of their restrictions away from $s$. $\mathrm{res}$ is the restriction of the bundles along the map $\mathrm{Spec}(\mathbb{F}_q[[t]]) \rightarrow \mathbb{P}^1$. For $A \in \mathcal{H}$, the Hecke operator $A^s:C[\mathrm{Bun}_G(\mathbb{P}^1,S)] \rightarrow C[\mathrm{Bun}_G(\mathbb{P}^1,S)]$ is defined as

$$A^s f={\pi_2}_!({\mathrm{res}}^*A \otimes {\pi_1}^* f)$$

\noindent The operator $A^s$ is independent of the choice of uniformizer. For $w \in W^{aff}$, the following is true.

\begin{centering}

\begin{tikzcd}
\mathrm{Bun}_G(\mathbb{P}^1,S)  & \mathrm{Corr}_w^s  \arrow[r,"\pi_2"] \arrow[l,"\pi_1"] & \mathrm{Bun}_G(\mathbb{P}^1,S) 
\end{tikzcd}

\end{centering}

$$T_w^s f={\pi_2}_! {\pi_1}^* f$$

\noindent $\mathrm{Corr}_w^s$ is the subspace of $\mathrm{Corr}^s$ where $\mathcal{E}_1$ and $\mathcal{E}_2$ restricted to a formal disk around $s$ are in relative position $w$.

\begin{definition}[Simultaneous Modification at Marked Points]
For $A \in \mathcal{H}$ and $R \subset S$, $A^R$ will denote the product of the Hecke operators $A^s$ for $s \in R$.

$$ A^R:= \prod_{s \in R} A^s$$

\noindent This is a product of commuting operators so the order of the product doesn't matter.
\end{definition}

\subsubsection{Reflection Operators}

\begin{definition}[Reflection Operator]
For simple reflections $s_{\alpha} \in W$ and $s \in S$, define the \textit{reflection operator} $\mathrm{Avg}^s_{s_{\alpha}}:=1+T_{s_{\alpha}}^s$.
\end{definition}

$\mathrm{Avg}^s_{s_{\alpha}}$ has the following interpretation. Let $P_{s_\alpha}$ denote the almost minimal parabolic corresponding to the simple coroot $\alpha$.

Let $\mathrm{Bun}_G(\mathbb{P}^1, S,s,s_{\alpha})$ be the moduli stack of $G$-bundles on $\mathbb{P}^1$ with Borel reduction at $S \setminus \{s\}$ and $P_{s_{\alpha}}$ reduction at $s$. For example, for $G=GL_n$, it classifies pairs $(\mathcal{E},\{F_p\}_{p \in S})$,

\begin{itemize}

\item $\mathcal{E}$ is a rank $n$ vector bundle on $\mathbb{P}^1$
\item For $p \neq s$, $F_p$ is a full flag in the fiber $\mathcal{E}|_p$
\item $F_s$ is an almost full flag in the fiber $\mathcal{E}|_s$, consisting of a space of each dimension except the one corresponding to $s_{\alpha}$.

\end{itemize}

There is a map $\pi:\mathrm{Bun}_G(\mathbb{P}^1, S) \rightarrow \mathrm{Bun}_G(\mathbb{P}^1, S,s,s_{\alpha})$. For $F \in C_{Aut}$, 

$$\mathrm{Avg}^s_{s_{\alpha}} \cdot F=\pi^* \pi_! F$$

\subsection{Pseudo-Eisenstein Series}

Given a compactly supported function $f:\Lambda \rightarrow \mathbb{C}$, the pseudo-Eisenstein series $\mathrm{Eis}_f$ is defined by the following induction diagram.
$$\Lambda \otimes \mathrm{Pic}(\mathbb{P}^1)\cong \mathrm{Bun}_T(\mathbb{P}^1) \xleftarrow{p} \mathrm{Bun}_B(\mathbb{P}^1) \xrightarrow{q} \mathrm{Bun}_G(\mathbb{P}^1,S)$$
$$\mathrm{Eis}_f=q_!p^*f$$

\noindent $p$ is the map associating the induced $T$-bundle to a $B$-bundle. $q$ is the map that associates the induced $G$-bundle and remembers the $B$ structure along $S$. Define $\mathrm{Bun}_B^{\lambda}(\mathbb{P}^1)$ as the preimage of the component $\lambda \in \mathrm{Bun}_T(\mathbb{P}^1)$ and $q_{\lambda}:\mathrm{Bun}_B^{\lambda}(\mathbb{P}^1) \rightarrow \mathrm{Bun}_G(X,S)$ the restriction of $q$. then

$$\mathrm{Eis}_{\lambda}:=\mathrm{Eis}_{\underline{1}_{\lambda}}={q_{\lambda}}!\underline{1}$$

Pseudo-Eisenstein series form a subpace of $C_{Aut}$. It is closed under spherical Hecke operators at $p \notin S$ but not under affine Hecke operators.

\begin{definition}[Eisenstein Module]
    The Eisenstein module, $C_{Eis}$, is the subspace of $C_{Aut}$ generated by the action of all affine Hecke operators on all pseudo-Eisenstein series.
\end{definition}

\noindent The space of Eisenstein series is also closed under spherical Hecke operators.

\subsubsection{Compatibility of Eisenstein series and Translation}

We describe the standard compatibility of Eisenstein induction with Hecke operators.

\begin{thm}\label{heckeeiscompatibility}
For compactly supported $f:\Lambda \rightarrow \mathbb{C}$ and $\mu \in \Lambda$, 

$$J_{\mu}^s \cdot \mathrm{Eis}_f=\mathrm{Eis}_{\mu \cdot f}$$

where $\mu \cdot f(\lambda)=f(\lambda-\mu)$.

\end{thm}

\begin{proof}
It suffices to show $J_{\mu}^s \mathrm{Eis}_{\lambda}=\mathrm{Eis}_{\lambda+\mu}$ for $\mu,\lambda \in \Lambda$ with $\mu$ anti-dominant. In this case $J_{\mu}=T_{\mu}$. We show that there is a diagram, where the left square is Cartesian and the upper left horizontal arrow is a homemorphism:

\begin{center}
\begin{tikzcd}
\mathrm{Bun}_{B}^{\lambda}(X) \arrow[d,"q_{\lambda}"]  & \Gamma \arrow[l,"t_1"] \arrow[d,"t_2"]  \arrow[r,"\cong"]& \mathrm{Bun}_B^{\lambda+\mu}(X) \arrow[d,"q_{\lambda+\mu}"] \\
\mathrm{Bun}_G(X,S) & \mathrm{Corr}_{\mu}^s \arrow[l,"\pi_1"] \arrow[r,"\pi_2"] & \mathrm{Bun}_G(X,S)
\end{tikzcd}
\end{center}

\noindent Assuming such a diagram exists,

$$\mathrm{J}_{\mu}^s \cdot \mathrm{Eis}_{\mu}={\pi_2}_!\pi_1^*{q_{\lambda}}_! \underline{1}={\pi_2}_!{t_2}_!t_1^* \underline{1}={q_{\lambda+\mu}}_!\underline{1}=\mathrm{Eis}_{\lambda+\mu}$$

\noindent The existence of such a diagram is shown in Lemma 2.4.4 of \cite{Nadler_2019}.

\end{proof}

\section{Example:  G=PGL(2)}

Fix $G=\mathrm{PGL}(2)$. Identify $\Lambda \cong \mathbb{Z}$ by $(t \mapsto \mathrm{diag}(t^k,1)) \mapsto k$. First, we first describe the geometry of $\mathrm{Bun}_G(\mathbb{P}^1,S)$ and compute the finite Hecke action on $C_{Aut}$ in the geometric basis of points of the moduli space. Then, we calculate the structure of $C_{Aut}$ as a $\mathcal{H}^{fin}$ trimodule. In this case $C_{Aut}=C_{Eis}$.
Finally, we will prove Theorem \ref{pgl2thm} characterizing $C_{Aut}$ as a $\mathcal{H}$ trimodule and confirm Conjecture \ref{mainconj} for $\mathrm{PGL}(2)$. 

\subsection{Finite Hecke Action}

There is a unique simple reflection. $\mathcal{H}^{fin}$ is generated by the operator $\mathrm{Avg}=1+T_{s_{\alpha}}$, which satisfies the quadratic relation $\mathrm{Avg}\cdot \mathrm{Avg}=(q+1) \mathrm{Avg}$. We compute the action of $\mathcal{H}^{fin}$ at $0 \in S$. The formulas for the action at other points are completely analogous.

We organize the calculation according to the following maps, given by forgetting parabolic structure.  $$\mathrm{Bun}_G(\mathbb{P}^1,S) \xrightarrow{\pi^0} \mathrm{Bun}_G(\mathbb{P}^1,\{1,\infty\}) \rightarrow \mathrm{Bun}_G(\mathbb{P}^1)$$

\noindent Recall that $\mathrm{Avg}^0=(\pi^0)^*{\pi^0}_!$. We list the rational points of $\mathrm{Bun}_G(\mathbb{P}^1,S)$ and record the fibers of the map $\pi^0$, so that we can compute the operator $\mathrm{Avg}^0$. We organize the information by fibers of the projection of $\mathrm{Bun}_G(\mathbb{P}^1)$.

There is a short exact sequence, $1 \rightarrow \mathbb{G}_m \rightarrow \mathrm{GL}_2 \rightarrow G \rightarrow 1$, so by the vanishing of the Brauer group of a curve,

$$\mathrm{Vect}_2(\mathbb{P}^1)/\mathrm{Pic}(\mathbb{P}^1) \cong \mathrm{Bun}_G(\mathbb{P}^1)$$

\noindent An of object of $\mathrm{Bun}_G(\mathbb{P}^1)$ is represented by a rank $2$ vector bundle, $\mathcal{E}$, up to tensoring with a line bundle. An object of $\mathrm{Bun}_G(\mathbb{P}^1,R)$, for $R \subset S$ is represented by a rank $2$ vector bundle, $\mathcal{E}$, up to tensoring with  line bundle, and a line $\ell_s$ in the fiber $\mathcal{E}|_s$, for $s \in R$.

\subsubsection{$\mathcal{E} \cong \mathcal{O} \oplus \mathcal{O}$} The first column records the automorphism group of the object. The next two columns record the poset of points of $\mathrm{Bun}_G(\mathbb{P}^1,S)$ and $\mathrm{Bun}_G(\mathbb{P}^1,\{1,\infty\})$, respectively. $x \rightarrow y$ means $y$ lies in the closure of $x$. The fibers of $\pi^0$ are indicated by color.

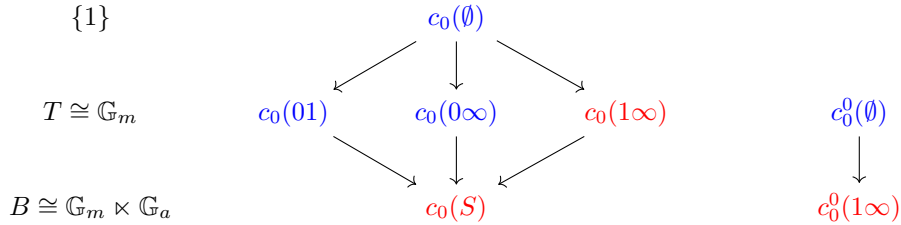
\begin{centering}

\begin{tikzcd}
\{ 1 \} &  & \textcolor{blue}{c_0(\emptyset)} \arrow[ld] \arrow[d] \arrow[rd] &  & & \\
  T \cong \mathbb{G}_m &  \textcolor{blue}{c_0(01)} \arrow[rd]& \textcolor{blue}{c_0(0 \infty)} \arrow[d] & \textcolor{red}{c_0 (1 \infty)} \arrow[ld] &  & \textcolor{blue}{c_0^0(\emptyset)} \arrow[d] \\
   B \cong \mathbb{G}_m \ltimes \mathbb{G}_a &  & \textcolor{red}{c_0(S)} &  & & 
   \textcolor{red}{c_0^0(1 \infty)}
\end{tikzcd}

\end{centering}

\noindent Identify the fibers $\mathcal{E}|_s$ for $s \in S$. For $R \subset S$, $c_0(R)$ denotes the locus where $\ell_s$ coincide for $s \in R$. Similarly, for $R \subset \{1,\infty\}$ $c_0^0(R)$ denotes the locus where $\ell_s$ coincide for $s \in R$. $\mathrm{Avg}^0$ acts as follows.

$$\mathrm{Avg}^0 \underline{1}_{c_0(S)}=(\pi^0)^* \pi^0_! {1}_{c_0(S)}=(\pi^0)^* \underline{1}_{c_0^0(S)} \frac{|\mathrm{Aut}(c_0^0(1\infty))|}{|\mathrm{Aut}(c_0(S))|}=\underline{1}_{c_0(S)}+\underline{1}_{c_0(1 \infty)}$$
$$\mathrm{Avg}^0 \underline{1}_{c_0(1 \infty)}=(\pi^0)^* \pi^0_! {1}_{c_0(1 \infty)}=(\pi^0)^* \underline{1}_{c_0^0(S)} \frac{|\mathrm{Aut}(c_0^0(1\infty))|}{|\mathrm{Aut}(c_0(1 \infty))|}=q\underline{1}_{c_0(S)}+q\underline{1}_{c_0(1 \infty)}$$

$$ \mathrm{Avg}^0 \underline{1}_{c_0(01)}=(\pi^0)^* \pi^0_! {1}_{c_0(01)}=(\pi^0)^* \underline{1}_{c_0^0(\emptyset)} \frac{|\mathrm{Aut}(c_0^0(\emptyset))|}{|\mathrm{Aut}(c_0(01))|}=\underline{1}_{c_0(01)}+\underline{1}_{c_0(0 \infty)}+\underline{1}_{c_0(\emptyset)}$$

$$ \mathrm{Avg}^0 \underline{1}_{c_0(0\infty)}=(\pi^0)^* \pi^0_! {1}_{c_0(0\infty)}=(\pi^0)^* \underline{1}_{c_0^0(\emptyset)} \frac{|\mathrm{Aut}(c_0^0(\emptyset))|}{|\mathrm{Aut}(c_0(0\infty))|}=\underline{1}_{c_0(01)}+\underline{1}_{c_0(0 \infty)}+\underline{1}_{c_0(\emptyset)}$$

$$ \mathrm{Avg}^0 \underline{1}_{c_0(\emptyset)}=(\pi^0)^* \pi^0_! {1}_{c_0(\emptyset)}=(\pi^0)^* \underline{1}_{c_0^0(\emptyset)} \frac{|\mathrm{Aut}(c_0^0(\emptyset))|}{|\mathrm{Aut}(c_0(\emptyset))|}=(q-1)\underline{1}_{c_0(01)}+(q-1)\underline{1}_{c_0(0 \infty)}+(q-1)\underline{1}_{c_0(\emptyset)}$$

\subsubsection{$\mathcal{E} \cong \mathcal{O}(1) \oplus \mathcal{O}$} We use the same conventions as before.

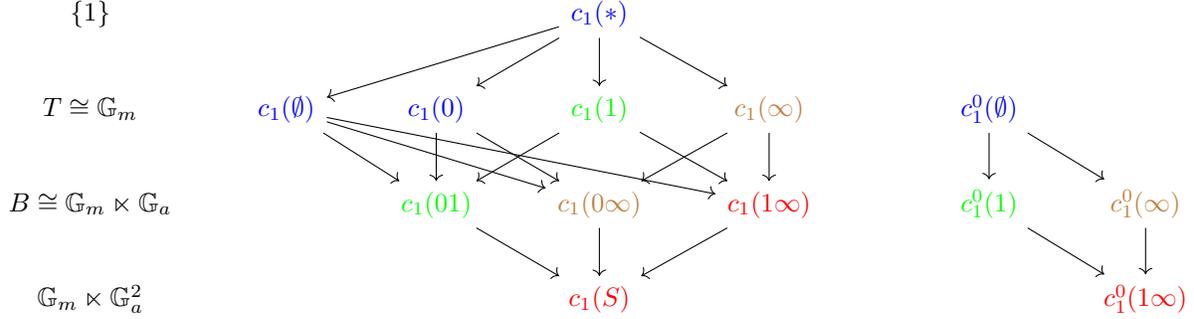
\begin{centering}

\begin{tikzcd}
\{ 1 \} & & & \textcolor{blue}{c_1(*)} \arrow[lld] \arrow[ld] \arrow[d] \arrow[rd] &  & & &\\
  T \cong \mathbb{G}_m & \textcolor{blue}{c_1(\emptyset)} \arrow[rd] \arrow[rrd] \arrow[rrrd]& \textcolor{blue}{c_1(0)} \arrow[d] \arrow[rd]& \textcolor{green}{c_1(1)} \arrow[ld] \arrow[rd] & \textcolor{brown}{c_1 (\infty)} \arrow[ld] \arrow[d] &  & \textcolor{blue}{c_1^0(\emptyset)} \arrow[d] \arrow[rd] & \\
   B \cong \mathbb{G}_m \ltimes \mathbb{G}_a & & \textcolor{green}{c_1(0 1)} \arrow[rd] & \textcolor{brown}{c_1(0 \infty)} \arrow[d] & \textcolor{red}{c_1(1 \infty)} \arrow[ld] & & 
   \textcolor{green}{c_1^0(1 )} \arrow[rd] & \textcolor{brown}{c_1^0(\infty )} \arrow[d]\\
    \mathbb{G}_m  \ltimes \mathbb{G}_a^2 & &  & \textcolor{red}{c_1(S)} &  & & 
   &\textcolor{red}{c_1^0(1 \infty)}
\end{tikzcd}
%\arrow[r,"\pi_2"] \arrow[l,"\pi_1"] 

\end{centering}

\noindent For $R \subset S$, $c_1(R)$ denotes the locus where there is a sub-bundle $\mathcal{O} \subset \mathcal{E}$, such that $\ell_s$ is contained in $\mathcal{O}(1)$ for $s \in R$ and $\ell_s$ is contained in the $\mathcal{O}$ sub-bundle for $s \notin R$. $c_1(*)$ is the generic configuration where no line $\ell_s$ lies in $\mathcal{O}(1)$ and there is no sub-bundle $\mathcal{O} \subset \mathcal{E}$ whose image contains the lines all the lines $\ell_s$. For $R \subset \{1,\infty\}$, $c_1^0(R)$ is the locus where the line $\ell_s$ is contained in $\mathcal{O}(1)$ if and only if $s \in R$. The action of $\mathrm{Avg}^0$ is as follows. We omit some of the intermediate computations.

$$\mathrm{Avg}^0 \underline{1}_{c_1(S)} = q^{-1} \mathrm{Avg}^0 \underline{1}_{c_1(1\infty)}=\underline{1}_{c_1(S)}+\underline{1}_{c_1(1\infty)}$$
$$\mathrm{Avg}^0 \underline{1}_{c_1(01)} = q^{-1} \mathrm{Avg}^0 \underline{1}_{c_1(1)}=\underline{1}_{c_1(01)}+\underline{1}_{c_1(1)}$$
$$\mathrm{Avg}^0 \underline{1}_{c_1(0 \infty)} = q^{-1} \mathrm{Avg}^0 \underline{1}_{c_1(\infty)}=\underline{1}_{c_1(0 \infty)}+\underline{1}_{c_1(\infty)}$$
$$\mathrm{Avg}^0 \underline{1}_{c_1(\emptyset)} =  \mathrm{Avg}^0 \underline{1}_{c_1(0)}=(q-1)^{-1} \mathrm{Avg}^0 \underline{1}_{c_1(*)}=\underline{1}_{c_1(\emptyset)}+\underline{1}_{c_1(0)}+\underline{1}_{c_1(*)}$$

\subsubsection{$\mathcal{E} \cong \mathcal{O}(k) \oplus \mathcal{O}$, $k \geq 2$} We use the same conventions as before.

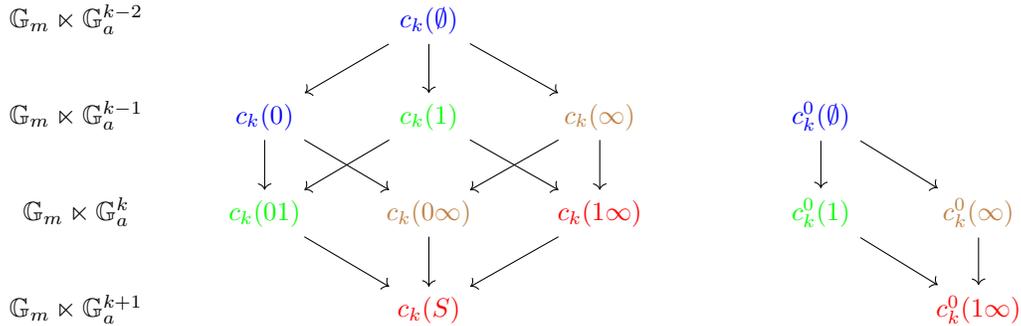
\begin{centering}

\begin{tikzcd}
\mathbb{G}_m \ltimes \mathbb{G}_a^{k-2} & & \textcolor{blue}{c_k(\emptyset)} \arrow[ld] \arrow[d] \arrow[rd] &  & & &\\
 \mathbb{G}_m \ltimes \mathbb{G}_a^{k-1} &  \textcolor{blue}{c_k(0)} \arrow[d] \arrow[rd]& \textcolor{green}{c_k(1)} \arrow[ld] \arrow[rd] & \textcolor{brown}{c_k (\infty)} \arrow[ld] \arrow[d] &  & \textcolor{blue}{c_k^0(\emptyset)} \arrow[d] \arrow[rd] & \\
  \mathbb{G}_m \ltimes \mathbb{G}_a^{k} & \textcolor{green}{c_k(0 1)} \arrow[rd] & \textcolor{brown}{c_k(0 \infty)} \arrow[d] & \textcolor{red}{c_k(1 \infty)} \arrow[ld] & & 
   \textcolor{green}{c_k^0(1 )} \arrow[rd] & \textcolor{brown}{c_k^0(\infty )} \arrow[d]\\
    \mathbb{G}_m \ltimes \mathbb{G}_a^{k+1} & & \textcolor{red}{c_k(S)} &  & & 
   &\textcolor{red}{c_k^0(1 \infty)}
\end{tikzcd}
%\arrow[r,"\pi_2"] \arrow[l,"\pi_1"] 

\end{centering}

\noindent For $R \subset S$, $c_k(R)$ denotes the locus where $\ell_s$ is contained in $\mathcal{O}(k)$ if and only if $s \in R$. For $R \subset \{1,\infty\}$, $c_k^0(R)$ is the locus where the line $\ell_s$ is contained in $\mathcal{O}(1)$ if and only if $s \in R$. The action of $\mathrm{Avg}^0$ is given by the following formula. For $R \subset \{1, \infty\}$,

$$\mathrm{Avg}^0 \underline{1}_{c_k(R \cup \{0\})}=q^{-1} \mathrm{Avg}^0 \underline{1}_{c_k(R )} = \underline{1}_{c_k(R \cup \{0\})}+\underline{1}_{c_k(R )}$$

\subsection{Finite Hecke Trimodule Structure}
\begin{definition}
For $k \in \Lambda_+$, let $C_{Aut}^k \subset C_{Aut}$ denote the subspace of functions that take nonzero values only on points lying over the bundle type $k \in \mathrm{Bun}_G(\mathbb{P}^1)$.
\end{definition}

$C_{Aut}^k$ is closed under finite Hecke operators at any $s \in S$. The space of automorphic functions admits the following decomposition into $\mathcal{H}^{fin}$ trimodules.

$$C_{Aut}=\oplus_{k \geq 0} C_{Aut}^k$$

The calculations of the previous section imply the following theorem.

\begin{prop}\label{pgl2relation}
    $C_{Aut}$ is the $(\mathcal{H}^{fin})^{\otimes S}$ module generated by $\{\underline{1}_{c_k(S)}\}_{k \geq 0} \cup \{\underline{1}_{c_1(\emptyset)}\}$ with the relations

\begin{equation}\label{triveq}
    \mathrm{Avg}^{\{0,1\}}  \underline{1}_{c_0(S)}= \mathrm{Avg}^{\{0,\infty\}} \underline{1}_{c_0(S)}=\mathrm{Avg}^{\{1,\infty\}} \underline{1}_{c_0(S)}
\end{equation}

\begin{equation}\label{triveq2}
    \mathrm{Avg}^s \underline{1}_{c_1(\emptyset)}=\mathrm{Avg}^sT_{s_{\alpha}}^{S \setminus \{s\}} \underline{1}_{c_1(S)}, \ \mathrm{for} \ s \in S
\end{equation}

\end{prop}

\begin{proof}

 For $k \geq 2$ and $R \subset S$,
$$ T_{s_{\alpha}}^R \underline{1}_{c_k(S)}=\underline{1}_{c_k(S \setminus R)}.$$

In particular, $C_{Aut}^k$ is freely generated by $c_k(S)$.

$C_{Aut}^0$ is generated by $\underline{1}_{c_0(S)}$. Check that $\mathrm{Avg}^{\{0,1\}}$ is the constant function on the locus where the bundle is trivial. Relation \ref{triveq} follows. It is easy to see that there are no other relations.

We check that $C_{Aut}^1$ is generated by $c_1(*)$ and $c_1(S)$ generate $C_{Aut}^1$ with a relations given by Equation \ref{triveq2}. First check that equation \ref{triveq2} is true using the calculations from the previous section. To see that relations are sufficient, observe that $\underline{1}_{c_1(S)}$ generates a free rank one $(\mathrm{H}^{fin})^{\otimes S}$ submodule of $C_{Aut}^1$ consisting of functions, $f$, satisfying $f(c_1(*))=f(c_1(\emptyset))$. This submodule has codimension one in $C^1_{Aut}$.
\end{proof}

\subsection{Hecke Trimodule Structure}

We state and prove Theorem \ref{pgl2thm}, confirming Conjecture \ref{mainconj} in this case. First we identify the Eisenstein functions.
\subsubsection{Eisenstein Objects}
\begin{prop}[Eisenstein Objects]\label{pgl2eis}
    $\mathrm{Eis}_k=\underline{1}_{c_k(S)}$ for $k \geq 0$ and $\mathrm{Eis}_{-1}=\underline{1}_{c_1(\emptyset)}$.
\end{prop}

\begin{proof}
Recall the induction diagram 
$$\mathrm{Bun}_T(\mathbb{P}^1) \leftarrow \mathrm{Bun}_B(\mathbb{P}^1) \rightarrow \mathrm{Bun}_G(\mathbb{P}^1,S)$$

\noindent Objects of $\mathrm{Bun}_B(\mathbb{P}^1)$ are represented by pairs $(\mathcal{L},\mathcal{E})$, $\mathcal{E}$ a rank $2$ vector bundle, and $\mathcal{L} \subset \mathcal{E}$ a rank $1$ sub-bundle, up to tensoring with a line bundle. The fiber above $k \in \Lambda$, $\mathrm{Bun}^k_B(\mathbb{P}^1)$, is the locus of pairs $(\mathcal{L},\mathcal{E})$, where $\mathcal{L} \cong \mathcal{O}(k)$ and $\mathcal{E}/\mathcal{L} \cong \mathcal{O}$. $\mathrm{Eis}_k$ is the pushforward of the constant function on $\mathrm{Bun}^k_B(\mathbb{P}^1)$.

Fix $k \geq -1$. We show that the image of $\mathrm{Bun}^k_B(\mathbb{P}^1)$ in $\mathrm{Bun}_G(\mathbb{P},S)$ is a single point. Suppose we have a short exact sequence of vector bundles bundles,

    $$\mathcal{O}(k) \rightarrow \mathcal{E} \rightarrow \mathcal{O},$$

\noindent $\mathrm{Ext}(\mathcal{O}(-k),\mathcal{O})=0$, so the short exact sequence splits. Therefore, $\mathrm{Bun}^k_B(\mathbb{P}^1,S)$ has a single point. If $k \geq 0$ the image of that point in $\mathrm{Bun}_G(\mathbb{P}^1,S)$ is $c_k(S)$, and if $k=-1$, the image is $c_1(\emptyset)$. There are three cases.

\begin{enumerate}
    \item For $k >0 $, comparing stabilisers, we find $\mathrm{Aut}((\mathcal{L},\mathcal{E})) \cong \mathbb{G}_m \ltimes \mathbb{G}_a^{k+1} \cong \mathrm{Aut}(c_k(S))$. It follows that $\mathrm{Eis}_k=\underline{1}_{c_k(S)}$.
    \item For $k=0$, $\mathrm{Aut}((\mathcal{L},\mathcal{E})) \cong B \cong \mathrm{Aut}(c_0(S))$. It follows that $\mathrm{Eis}_0=\underline{1}_{c_0(S)}$.
    \item Finally, for $k=-1$, $\mathrm{Aut}((\mathcal{L},\mathcal{E})) \cong T \cong \mathrm{Aut}(c_1(\emptyset))$. $\mathrm{Eis}_{-1}=\underline{1}_{c_1(\emptyset)}$. 

\end{enumerate}

It is not necessary for the following calculations but one can calculate $\mathrm{Eis}_k$ for $k \leq -2$. For example
$$\mathrm{Eis}_{-2}=\underline{1}_{c_0(\emptyset)}+\underline{1}_{c_2(\emptyset)}.$$

\noindent In general, $\mathrm{Eis}_k$ for $k \leq -2$ is nonzero only on points of moduli space where the bundle is $\mathcal{E} \cong \mathcal{O}(r) \oplus \mathcal{O}$ with $0 \leq r \leq -k$ the same parity as $k$.

\end{proof}

\subsubsection{Main Theorem}\label{proofmainthm}

\begin{thm}\label{pgl2thm}

$C_{Aut}$ is the $\mathcal{H}^{\otimes S}$ module generated by $\mathrm{Eis}_0$ with the relations

\begin{equation}\label{pgl2thmeq1}
J_k^0 \mathrm{Eis}_0=J_k^1 \mathrm{Eis}_0=J_k^{\infty} \mathrm{Eis}_0 \ \text{for} \ k \in \Lambda
\end{equation}

\begin{equation}\label{pgl2thmeq2}
\mathrm{Avg}^{\{0,1\}}  \mathrm{Eis}_0= \mathrm{Avg}^{\{0,\infty\}}  \mathrm{Eis}_0=\mathrm{Avg}^{\{1,\infty\}}  \mathrm{Eis}_0
\end{equation}

\end{thm}

\begin{proof}
    By Proposition \ref{pgl2relation}, $C_{Aut}$ is generated by Eisenstein functions under Hecke operators. By Theorem \ref{heckeeiscompatibility}  all Eisenstein functions are generated by $\mathrm{Eis}_0$ under translation Hecke operators. Therefore, $C_{Aut}$ is generated by $\mathrm{Eis}_0$.
    
    We check that the stated relations hold. Equation \ref{pgl2thmeq1} is a consequence Theorem \ref{heckeeiscompatibility} on compatibility of translation Hecke operators with Eisenstein induction. By Proposition \ref{pgl2eis}, $\mathrm{Eis}_0=\underline{1}_{c_0(S)}$, so Equation \ref{pgl2thmeq2} follows from Equation \ref{triveq} of Proposition \ref{pgl2relation}.

    We show that there are no other relations. Let $\widetilde{C}$ denote the quotient of $\mathcal{H}^{\otimes S}$ by the left ideal generated by the relations stated in Equations \ref{pgl2thmeq1} and \ref{pgl2thmeq2}. There is a surjection of $\mathcal{H}^{\otimes S}$ modules

    $$\widetilde{C} \rightarrow C_{Aut}$$

    \noindent We will show that this an injective map of $(\mathcal{H}^{fin})^{\otimes S}$ modules. Let $\widetilde{C}_+ \subset \widetilde{C}$ be the $(\mathcal{H}^{fin})^{\otimes S}$ submodule generated by $\{J_k^0\}_{k \geq -1}$. By Proposition \ref{pgl2relation} and Proposition \ref{pgl2eis}, it suffices to show that the following are true in $\widetilde{C}$:

    \begin{equation}
  \mathrm{Avg}^s J_{-1}^0=\mathrm{Avg}^sT_{s_{\alpha}}^{S \setminus \{s\}} J_1, \ \mathrm{for} \ s \in S
    \end{equation}

    \begin{equation}
    J_k \in \widetilde{C}_+ \ \text{for} \ k \leq -2  
    \end{equation}

\noindent We have omitted the superscript, $s \in S$, on the operators $J_k$ because of the defining relations of $\widetilde{C}$. The formulas follow from Proposition \ref{technicalprop}.
    
\end{proof}

\section{ Proof of Theorem \ref{mainthm}}

For this section, assume $G$ is such that $\rho \in \Lambda$.

\begin{definition}\label{formaleis}
The \textit{algebraic} Eisenstein module, $\widetilde{C}$, is the quotient $\mathcal{H}^S$ module which is the quotient of   $\mathcal{H}^S$ by the left ideal generated by the relations  

$$ J_\lambda^0=J_\lambda^1=J_{\lambda}^{\infty} \ \text{for} \ \lambda \in \Lambda$$

$$ \mathrm{Avg}_{s_{\alpha}}^{\{0,1\}}=\mathrm{Avg}_{s_{\alpha}}^{\{0,\infty\}}=\mathrm{Avg}_{s_{\alpha}}^{\{1,\infty\}} \ \text{for simple} \ s_{\alpha} \in W$$
(delete definition of $C_{Eis}$ everywhere). maybe put definition in background section?
\end{definition}

\begin{thm} \label{checkrelthm}
    There is a surjective map of $\mathcal{H}^{\otimes S} $ modules, $\widetilde{C} \rightarrow C_{Eis}$  given by $1 \mapsto \mathrm{Eis}_0$.
\end{thm}
\begin{proof}
This is equivalent to checking the Translation and Reflection relations on $C_{Eis}$. By  Theorem \ref{heckeeiscompatibility} $J_{\lambda}^s \mathrm{Eis}_{0}=\mathrm{Eis}_{\lambda}$ for any $s \in S$ and $\lambda \in \Lambda$. In particular, $J_{\lambda}^s \mathrm{Eis}_{0}$ is independent of $s$.

Fix a simple coroot $\alpha$. Let $P_{s_{\alpha}}$ the almost minimal parabolic associated with $s_{\alpha}$.  For $R \subset S$, let $\mathrm{Bun}_G(\mathbb{P}^1,S,R,s_{\alpha})$ denote the moduli of stack of $G$-bundles on $\mathbb{P}^1$ with Borel reductions at $S \setminus R$ and $P_{s_{\alpha}}$ reduction at $R$. There is a map $\pi:\mathrm{Bun}_G(\mathbb{P}^1,S) \rightarrow  \mathrm{Bun}_G(\mathbb{P}^1,\{0,1\},s_{\alpha})$ forgetting parabolic structure at $0$ and $1$. Consider the following diagram, where the square is Cartesian

\begin{centering}

\begin{tikzcd}
\mathrm{Bun}_G(\mathbb{P}^1,S) \arrow[d,  "\pi_0"] & & \\
\mathrm{Bun}_G(\mathbb{P}^1,S,\{0\},s_{\alpha}) \arrow[d,  "\pi_{01}"] & \mathrm{Bun}_G(\mathbb{P}^1,S) \arrow[l,  "\pi_0"],\arrow[d,  "\pi_1"]& \\
\mathrm{Bun}_G(\mathbb{P}^1,S,\{0,1\},s_{\alpha}) & \mathrm{Bun}_G(\mathbb{P}^1,S,\{1\},s_{\alpha}) \arrow[l,  "\pi_{10}"] & \mathrm{Bun}_G(\mathbb{P}^1,S) \arrow[l,  "\pi_1"]
\end{tikzcd}

\end{centering}

\noindent For $F \in C_{Aut}$,

$$\mathrm{Avg}^{\{0,1\}}_{s_{\alpha}}F=\mathrm{Avg}_{s_{\alpha}}^1\mathrm{Avg}_{s_{\alpha}}^0F=\pi_1^*{\pi_1}_{!}\pi_0^*\pi_0!F=\pi_1^*\pi_{10}^*{\pi_{01}}_!\pi_0!F=\pi^*\pi_!F$$

\noindent There is a point $\mathrm{pt}/B \rightarrow \mathrm{Bun}(\mathbb{P}^1,S)$ classifying trivial bundles with the same Borel reduction at all points of $S$. There is also a point $\mathrm{pt}/B \rightarrow \mathrm{Bun}_G(\mathbb{P}^1,\{0,1\},s_{\alpha})$ classifying trivial bundles with the same Parabolic structure at all points of $S$ (and the unique, up to automorphism, of the further reduction of the structure group to $B$ at $\infty$). The following diagram commutes

\begin{centering}

\begin{tikzcd}
    \mathrm{pt}/B \arrow[r] \arrow[rd] & \mathrm{Bun}_G(\mathbb{P}^1,S) \arrow[d,"\pi"] \\
     & \mathrm{Bun}_G(\mathbb{P}^1,\{0,1\},s_{\alpha})
\end{tikzcd}

\end{centering}

\noindent Therefore, 
$$\pi_!\mathrm{Eis}_0=\pi_! \underline{1}_{\mathrm{pt}/B}=\underline{1}_{\mathrm{pt}/B}$$
%%% %%%%

\noindent The fiber above $\mathrm{pt}/B$ of $\pi$ is the locus where the bundle is trivial and the Borel reductions at the points of $S$ have the same $P_{s_{\alpha}}$ reduction. $\mathrm{Avg}_{s_{\alpha}}^{\{0,1\}} \mathrm{Eis}_0=\pi^* \underline{1}_{\mathrm{pt}/B}$ is the constant function on this locus. By symmetry, we see that $\mathrm{Avg}_{s_{\alpha}}^{S \setminus\{s\}} \mathrm{Eis}_0$ is independent of $s$.

\end{proof}

$\widetilde{C}$ and $C_{Eis}$ are $\mathcal{H}^{\otimes}$ modules. By restriction of scalars through $\mathbb{C}[\Lambda] \rightarrow \mathcal{H}^0$, these become modules over algebra of translation operators at $0$. We make some observations about these modules.

\begin{prop} \label{fin}
$C_{Eis}$ is finitely generated over $\mathbb{C}[\Lambda]$.
\end{prop}
\begin{proof}
We show that $C_{Eis}$ is generated by the $|W|^3$ elements $T^0_{w_0}T^1_{w_1}T^{\infty}_{w_{\infty}}\mathrm{Eis}_0$ for $w_0,w_1,w_{\infty} \in W$. Recall that $\{ T_wJ_{\lambda}\}_{w \in W, \lambda \in \Lambda}$ is a basis for $\mathcal{H}$. Therefore, the Eisenstein module is spanned by functions elements

$$T_{w_0}^0J_{\lambda_0}^0T_{w_1}^1J_{\lambda_1}^1T_{w_{\infty}}^{\infty}J_{\lambda_{\infty}}^{\infty} \mathrm{Eis}_0,$$

\noindent $w_0,w_1,w_{\infty} \in \Lambda$ and $\lambda_0,\lambda_1,\lambda_{\infty} \in \Lambda$. By the translation relation

$$T_{w_0}^0J_{\lambda_0}^0T_{w_1}^1J_{\lambda_1}^1T_{w_{\infty}}^{\infty}J_{\lambda_{\infty}}^{\infty}=T_{w_0}^0T_{w_1}^1T_{w_{\infty}}^{\infty} J_{\lambda}^0 \mathrm{Eis}_0,$$

\noindent where $\lambda=\lambda_0+\lambda_1+\lambda_{\infty}$. because $\{J_{\lambda}T_w\}_{w \in W, \lambda \in \Lambda}$ is another basis for $\mathcal{H}$, $C_{Eis}$ is spanned by functions
$$J_{\lambda}^0T_{w_0}^0T_{w_1}^1T_{w_{\infty}}^{\infty}  \mathrm{Eis}_0.$$

\noindent In particular, $C_{Eis}$ is generated over $\mathbb{C}[\Lambda]$ by functions $T^0_{w_0}T^1_{w_1}T^{\infty}_{w_{\infty}}\mathrm{Eis}_0$.

\end{proof}

\begin{prop} \label{geo}
    $C_{\mathrm{Eis}}$ contains a free $\mathbb{C}[\Lambda]$ submodule of rank $|W|^2$.
\end{prop}
\begin{proof}
    We claim that $|W|^2$ elements $T^1_{w_1}T^{\infty}_{w_{\infty}}\mathrm{Eis}_0$ are independent over $\mathbb{C}[\Lambda]$. Suppose there is some nontrivial finite linear combination

    $$\sum_i c_iJ_{\lambda_i}^0T_{w_{1,i}}^1T_{w_{\infty},i}^{\infty} \mathrm{Eis}_0=0$$

    \noindent Pick a weight $\lambda$ such that $\mu+\lambda_i-\rho \in \Lambda_+$. $J_{\lambda}$ is invertible, so

    $$\sum_i c_iJ_{\lambda_i}^0T_{w_{1,i}}^1T_{w_{\infty},i}^{\infty} \mathrm{Eis}_0=0 \iff \sum_i c_iJ_{\mu+\lambda_i}^0T_{w_{1,i}}^1T_{w_{\infty},i}^{\infty} \mathrm{Eis}_0=0 \iff \sum_i c_i T_{w_{1,i}}^1T_{w_{\infty},i}^{\infty} \mathrm{Eis}_{\mu+\lambda_i} $$

    \noindent In particular, it suffices to show that the functions $\{T_{w_{1,i}}^1T_{w_{\infty},i}^{\infty} \mathrm{Eis}_{\lambda}\}$, for $w_1,w_{\infty} \in W$ and $\lambda - \rho \in \Lambda_+$ are linearly independent.

    Identify isomorphism classes of $G$-bundles on $\mathbb{P}^1$ with $W \backslash W^{aff} / W \cong \Lambda_+$. If $\mathcal{E}_{\lambda}$ is a $G$ bundle corresponding to $\lambda \in \Lambda_+$ such that $\lambda-\rho \in \Lambda_+$, then there is a $B$-bundle $\mathcal{E}_{B,\lambda}$, stable under $\mathrm{Aut}(\mathcal{E}_{\lambda})$. In particular, for $s \in S$ there is a flag $F_s \subset \mathcal{E}_{\lambda}|_s$ that is stable under $\mathrm{Aut}(\mathcal{E}_{\lambda})$. The function $\{T_{w_{1,i}}^1T_{w_{\infty},i}^{\infty} \mathrm{Eis}_{\lambda}\}$ is supported only on points of the locus classifying parabolic bundles $(\mathcal{E},\{F_s'\}_{s \in S})$ where $\mathcal{E} \cong \mathcal{E}_{\lambda}$ and the parabolic structure at $s \in \{1,\infty\}$ is in relative position $w_s$ to $F_s'$.

\end{proof}
\begin{rmk}
If $\lambda-2\rho \in \Lambda_+$, then the isomorphism class of an object in $\mathrm{Bun}_{G}(\mathbb{P}^1,S)$ with underlying bundle $\mathcal{E}_{\lambda}$ is determined by the relative positions, for $s \in S$, of the parabolic structure $F_s'$ to $F_s$. We haven't proved this observation as it isn't needed for any of our results.
\end{rmk}

\begin{conj} \label{secondconj}

$\widetilde{C}$  is free of rank $|W|^2$ over $\mathbb{C}[\Lambda]$.

\end{conj}

\begin{ex}
    Conjecture \ref{secondconj} is true for $G=\mathrm{PGL}(2)$. $C_{Eis}$ is generated over $\mathbb{C}[\Lambda]$ by the following functions: $$\mathrm{Eis}_0,T^1 \mathrm{Eis}_0,T^{\infty} \mathrm{Eis}_0, T^{\{1, \infty\}} \mathrm{Eis}_0,T^0 \mathrm{Eis}_0$$
    
    \noindent The first four generators are independent over $\mathbb{C}[\Lambda]$. One can check that
    
\begin{equation}\label{pgl2basis}
(q^2J_2^0-1)(T^{1}T^{\infty}-qT_0)=(q-1)(T^1-q)(T^{\infty}-q)
\end{equation}

\noindent Therefore, $\{\mathrm{Eis}_0,T^1\mathrm{Eis}_0,T^{\infty}\mathrm{Eis}_0,(T^{\{1,\infty\}}-qT^0)\mathrm{Eis}_0\}$ is a basis over $\mathbb{C}[\Lambda]$.
\end{ex}

 Conjecture \ref{secondconj} together with Propositions \ref{fin} and \ref{geo} imply that $\widetilde{C} \rightarrow C_{Eis}$ is an isomorphism. We show that \ref{secondconj} is generically true over $\mathbb{C}[\Lambda]$.

\begin{prop} \label{rankprop}
$\mathrm{Frac}(\mathbb{C}[\Lambda]) \otimes_{\mathbb{C}[\Lambda]} \widetilde{C}$ has dimension $|W|^2$ over $\mathcal{K}^0$.

\end{prop}

\noindent Let us postpone the proof of Proposition \ref{rankprop} briefly. It will be easier to filter the vector space $\widetilde{C}$ and work with the associated graded vector space.

$\mathcal{H}$ is filtered by length $\ell:W \rightarrow \mathbb{Z}_{\geq 0}$. For $i \in \mathbb{Z}_{\geq 0}$ The $i$th filtered component $F^i(\mathcal{H}) \subset \mathcal{H}$ is spanned by $T_wJ_{\lambda}$ for $w \in W$ with $\ell(w) \leq i$ and $\lambda \in \Lambda$. $F^0(\mathcal{H}) \cong \mathbb{C}[\Lambda]$ is the subalgebra of translation operators. Note that $F^i(\mathcal{H})$ is also spanned by $J_{\lambda}T_w$ for $w \in W$ with $\ell(w) \leq i$ and $\lambda \in \Lambda$. In general,

$$F^i(\mathcal{H}) \cdot F^j(\mathcal{H}) \subset F^{i+j} (\mathcal{H})\ \forall \ i,j \in \mathbb{Z}_{\geq 0}$$

We filter $\widetilde{C}$ so that the following are true:

\begin{enumerate}
\item The $\widetilde{C}$ is a filtered module for the filtered algebra $\mathcal{H}^0$ of Hecke operators at $0$. That is,

$$F^i(\mathcal{H}^0) \cdot F^j(\widetilde{C}) \subset F^{i+j}\widetilde{C}, \ \forall \ i,j \in \mathbb{Z}_{\geq 0}$$

\item The filtration on $\widetilde{C}$ is preserved by Hecke operators at $S \setminus \{0\}$. That is, for $s \in S \setminus \{0\}$,

$$F^i(\mathcal{H}^s) \cdot F^j(\widetilde{C}) \subset F^{j}\widetilde{C}, \ \forall \ i,j \in \mathbb{Z}_{\geq 0}$$
\end{enumerate}

\noindent Note that a consequence of the first requirement is that the filtered components of $\widetilde{C}$ are modules for the algebra translation operators at $0$, $\mathbb{C}[\Lambda]$.

\begin{definition}[Filtration of $\widetilde{C}$] \label{eisfilt}
The $i$th filtered component $F^i(\widetilde{C}) \subset \widetilde{C}$ is spanned by $A^{0}T_{w_1}^1T_{w_{\infty}}^{\infty} $, for $w_1,w_{\infty} \in W$ and $A \in F^i(\mathcal{H})$. Alternatively, it is spanned by $T_{w_0}^0T_{w_1}^1T_{w_{\infty}}^{\infty} J_{\lambda}$ for $\lambda \in \Lambda$ and $w_0,w_1,w_{\infty} \in W$ with $\ell(w_0) \leq i$.

\end{definition}

\noindent The first requirement on the filtration of $\widetilde{C}$ is automatically satisfied by construction. The second condition is also satisfied  because of the translation relation. Now we prove Proposition \ref{rankprop}

\begin{proof}[Proof of Proposition \ref{rankprop}]

After rationalization, we have the isomorphism

$$\mathrm{Frac}(\mathbb{C}[\Lambda]) \otimes_{\mathbb{C}[\Lambda]} \widetilde{C} \cong \bigoplus_i \mathrm{Frac}(\mathbb{C}[\Lambda]) \otimes_{\mathbb{C}[\Lambda]} \mathrm{Gr}^i(\widetilde{C})$$

\noindent $\mathrm{Gr}^i(\widetilde{C})$ is generated as a $\mathbb{C}[\Lambda] \otimes \mathcal{H}^{\otimes \{1,\infty\}}$ module by $T_w^0$, for $w \in W$ with $\ell(w)=i$. Therefore, we need only to show that for $w \in W$ of length $\ell(w)=i$, there is $A \in \mathbb{C}[\Lambda]$ such that

$$A^0 \cdot T_w^0 \in F^{i-1}(\widetilde{C}).$$

\noindent Pick a simple reflection $s_{\alpha}$ such that $\ell(w s_{\alpha})=\ell-1$. Pick $\lambda \in \Lambda$ such that $\langle \check{\alpha},\lambda \rangle =1$. Start with the equation from Proposition \ref{technicalprop2},

$$T_{s_{\alpha}}^0(J_{\lambda}-J_{s_{\alpha}(\lambda)}) \in F^0 (\widetilde{C})$$

$$J_{-s_{\alpha}(\lambda)}^1T_{s_{\alpha}}^0(J_{\lambda}-J_{s_{\alpha}(\lambda)}) \in F^0 (\widetilde{C})$$

$$T_{s_{\alpha}}^0(J_{\alpha}-1) \in F^0 (\widetilde{C})$$

$$T_w^0(J_{\alpha}-1) \in F^0 (\widetilde{C}) \in F^{i-1}(\widetilde{C})$$

\noindent Observe that for some integer $n$, $T_wJ_{\alpha}-q^nJ_{w\cdot \alpha}T_w \in \mathbb{C}[\Lambda]$, so

$$(q^nJ_{w \cdot \alpha}^0-1)T_w^0 \in F^0 (\widetilde{C}) \in F^{i-1}(\widetilde{C})$$
    
\end{proof}

\begin{rmk}
    $w \cdot \alpha$ is always a negative coroot. $n$ is given by the explicity formula $n=\langle \check{\rho}, w \cdot \alpha \rangle -1$. For Proposition \ref{rankprop} we only needed to invert the polynomial
    $$\prod_{\alpha \in R_+} \left( q^{\langle \rho,\alpha \rangle +1}J_{\alpha}-1 \right) $$
    
    \noindent This is not the standard discriminant polynomial. In particular, $q^{\langle \rho,\alpha \rangle +1}J_{\alpha}-1$ is not homogeneous with respect to the natural $q$-twisted $\mathbb{G}_m$ action on $\mathbb{C}[\Lambda]$.
\end{rmk}

Theorem \ref{mainthm} follows from Propositions \ref{fin}, \ref{geo}, and \ref{rankprop}.

\section{Some Formulas for Algebraic Eisenstein Module}

In this section we prove some formulas that hold in the module $\widetilde{C}$ formally generated over $\mathcal{H}^{\otimes S}$ by one generator subject to the translation and reflection relations. We have postponed these calculations to this section as they don't fit the flow of the arguments where they are used. It is helpful to first understand the $G=\mathrm{PGL}(2)$ example.

\begin{prop}[Functional Equation for Algebraic Eisenstein Module]\label{technicalprop}
Let $\widetilde{C}$ be the quotient of $\mathcal{H}^{\otimes S}$ by the left ideal generated by relations:

$$ J_\lambda^0=J_\lambda^1=J_{\lambda}^{\infty} \ \text{for} \ \lambda \in \Lambda$$

$$ \mathrm{Avg}_{s_{\alpha}}^{\{0,1\}}=\mathrm{Avg}_{s_{\alpha}}^{\{0,\infty\}}=\mathrm{Avg}_{s_{\alpha}}^{\{1,\infty\}} \ \text{for simple} \ s_{\alpha} \in W$$

\noindent Assume that the map $\check{\alpha}:\Lambda \rightarrow \mathbb{Z}$ given by $\lambda \mapsto \langle \check{\alpha},\lambda \rangle$ is surjective. Then, for any simple reflection $s_{\alpha}$,

$$\mathrm{Avg}_{s_{\alpha}}^sJ_{\lambda} = \mathrm{Avg}_{s_{\alpha}}^sT_{s_{\alpha}}^{S \setminus \{s\}}J_{s_{\alpha}(\lambda)} \ \text{if} \ \langle \check{\alpha},\lambda\rangle=-1$$
  
$$J_{\lambda} \in  \mathrm{Span}_{(\mathcal{H}^{fin})^{\otimes S}} \{ J_{\mu}\}_{\mu \in R(\lambda,\alpha)} \ \text{if} \ \langle \check{\alpha},\lambda\rangle \leq -2$$

\noindent where $R(\lambda,\alpha) \subset \Lambda$ consists of coweights $\mu$, such that $\mu-\lambda$ is an integral multiple of $\alpha$ and $-1\leq \langle \check{\alpha},\mu\rangle  \leq -\langle \check{\alpha},\lambda\rangle  $.

\end{prop}

\begin{proof}

Fix the simple coroot $\alpha$ and let $T:=T_{s_{\alpha}}$, $\mathrm{Avg}:=\mathrm{Avg}_{s_{\alpha}}$.Fix $\lambda$ so that $\langle \check{\alpha},\lambda \rangle=1$. Start with the reflection relation

$$T^1\mathrm{Avg}^0=T^{\infty}\mathrm{Avg}^0$$

$$T^1J_{\lambda}^1T^1\mathrm{Avg}^0=T^1J_{\lambda}^1T^{\infty}\mathrm{Avg}^0$$

\noindent Observe that $TJ_{\lambda}T=J_{\lambda-\alpha}$.

$$\mathrm{Avg}^0J_{\lambda-\alpha}=\mathrm{Avg}^0T^{1 \infty}J_{\lambda}$$

\noindent This proves the first part of the proposition. Continuing with the previous equality

$$T^0J_{\lambda}^0\mathrm{Avg}^0J_{\lambda-\alpha}=T^0J_{\lambda}^0\mathrm{Avg}^0T^{1 \infty}J_{\lambda}$$

$$J_{2\lambda-2\alpha}+T^0J_{2 \lambda-\alpha}=T^SJ_{2 \lambda}+T^{1\infty}J_{2\lambda-\alpha}$$

$$J_{2 \lambda-2\alpha}=T^SJ_{2\lambda}+(T^{1\infty}-T^0)J_{2\lambda-\alpha}$$

\noindent Now,  let $\lambda' \in \Lambda$ be such that $\langle \check{\alpha},\lambda' \rangle \leq -2$. Define $\mu:=\lambda'-2\lambda$ and $n:=-\langle \check{\alpha},\mu \rangle \in \mathbb{Z}_{\geq 0}$.

$$J_{\lambda'}=J_{\mu}^0T^SJ_{2\lambda}+J_{\mu}^0(T^{1\infty}-T^0)J_{2\lambda-\alpha} \in \mathrm{Span}_{(\mathcal{H}^{fin})^{\otimes S}} \{J_{\lambda'+k\alpha}\}_{k=1}^{n}$$

The second part of the proposition follows by induction on $n$.

\end{proof}

\begin{prop}\label{technicalprop2}
Let $\widetilde{C}$ be as in Proposition \ref{technicalprop}. If $\lambda \in \Lambda$ such that $\langle \check{\alpha},\lambda\rangle =1$, then the following is true in $\widetilde{C}$:

$$T_{s_{\alpha}}^0 (J_{\lambda}-J_{s_{\alpha}(\lambda)})=-T_{s_{\alpha}}^{\{1,\infty\}}(J_{\lambda}-q^{-1}J_{s_{\alpha}(\lambda)})-(1+T_{s_{\alpha}}^1+T_{s_{\alpha}}^{\infty})q^{-1/2}J_{s_{\alpha}(\lambda)}$$
    
\end{prop}

\begin{proof}
For ease of notation, let $T:=T_{s_{\alpha}}$. Introduce the operator $D \in \mathcal{H}$,
$$D:=q^{1/2}J_{\lambda}-q^{-1/2}J_{s_{\alpha}(\lambda)}$$

\noindent Observe that

\begin{equation}\label{Deq}
DT=-TD+\frac{2(q-1)D}{1-qJ_{\alpha}}=-TD-2(q-1)q^{-1/2}J_{s_{\alpha}(\lambda)}
\end{equation}

Start with the reflection relation.

$$(1+T^1)(T^0 -T^{\infty})=0$$

$$\implies (D^1-D^0)(1+T^1)(T^0-T^{\infty})=0$$

\noindent Using Equation \ref{Deq} to move all $D$ operators to the right and simplifying we obtain the following. In light of the translation relation, the superscript is omitted from all translation that appear as the rightmost term of an expression.

$$(T^0+T^{\{1,\infty\}})D=(1+T^1 + T^{\infty}-T^0)\frac{(q-1)D}{1-qJ_{\alpha}} $$

$$ \implies T^0 \left( D+\frac{(q-1)D}{1-qJ_{\alpha}}\right)= -T^{1}T^{\infty} D+(1+T^1+T^{\infty})\frac{(q-1)D}{1-qJ_{\alpha}}$$

$$\implies q^{1/2}T^0 (J_{\lambda}-J_{s_{\alpha}(\lambda)}) =-T^{1}T^{\infty} D+(1+T^1+T^{\infty})\frac{(q-1)D}{1-qJ_{\alpha}} $$

\end{proof}

\section{Example:  G=SL(3)}
This section has been included to provide some evidence that the conjecture is true and give some intuition for the general structure of $C_{Eis}$. Fix $G=\mathrm{SL}(3)$ for this section. We will prove Theorem \ref{sl3thm} verifying Conjecture \ref{mainconj} in this case.

\begin{thm}\label{sl3thm}
Conjecture \ref{mainconj} is true when $G=\mathrm{SL}(3)$.
\end{thm}

Our approach is to study the Eisenstein module as a finite Hecke trimodule. It is not expected that this approach will generalize to arbitrary $G$.

Identify the coweight lattice $$
\Lambda \cong \{ (k_1,k_2,k_3) \in \mathbb{Z}^3 : \ k_1+k_2+k_3=0\}$$

\noindent by $ (t \mapsto \mathrm{diag}(t^{k_1},t^{k_2},t^{k_3})) \mapsto (k_1,k_2,k_3) $. There are two simple coroots, $\alpha_1=(1,-1,0)$ and $\alpha_2=(0,1,-1)$. $\rho=(1,0,-1)$. The Weyl group is identified $W \cong S_3$ with it's standard action on $\mathbb{Z}^3$. $s_{\alpha_i}$ is identified with the standard generator $s_i \in S_3$. Reflection normal to the long root is identified with $s_3 \in S_3$, $s_3=s_1s_2s_1=s_2s_1s_2$. To simplify notation, define $T_i$ and $\mathrm{Avg}_i \in \mathcal{H}$, for $i=1,2$ as $T_{i}=T_{s_{\alpha_i}}$ and $\mathrm{Avg}_i=\mathrm{Avg}_{s_{\alpha_i}}$.

Let $\widetilde{C}$ be the algbraic Eisenstein module as in Definition \ref{formaleis}. By Theorem \ref{checkrelthm} there is a surjective map of $\mathcal{H}^{\otimes S}$ modules $\widetilde{C} \rightarrow C_{Eis}$. By Proposition \ref{technicalprop}, $\widetilde{C}$ is generated over $(\mathcal{H}^{fin})^{\otimes S}$ by $J_{\lambda}$ for $\lambda \in \Lambda$ such that $\lambda+\rho$ is dominant. Further, the following relations hold amongst the generators (see Figure \ref{sl3fig}):

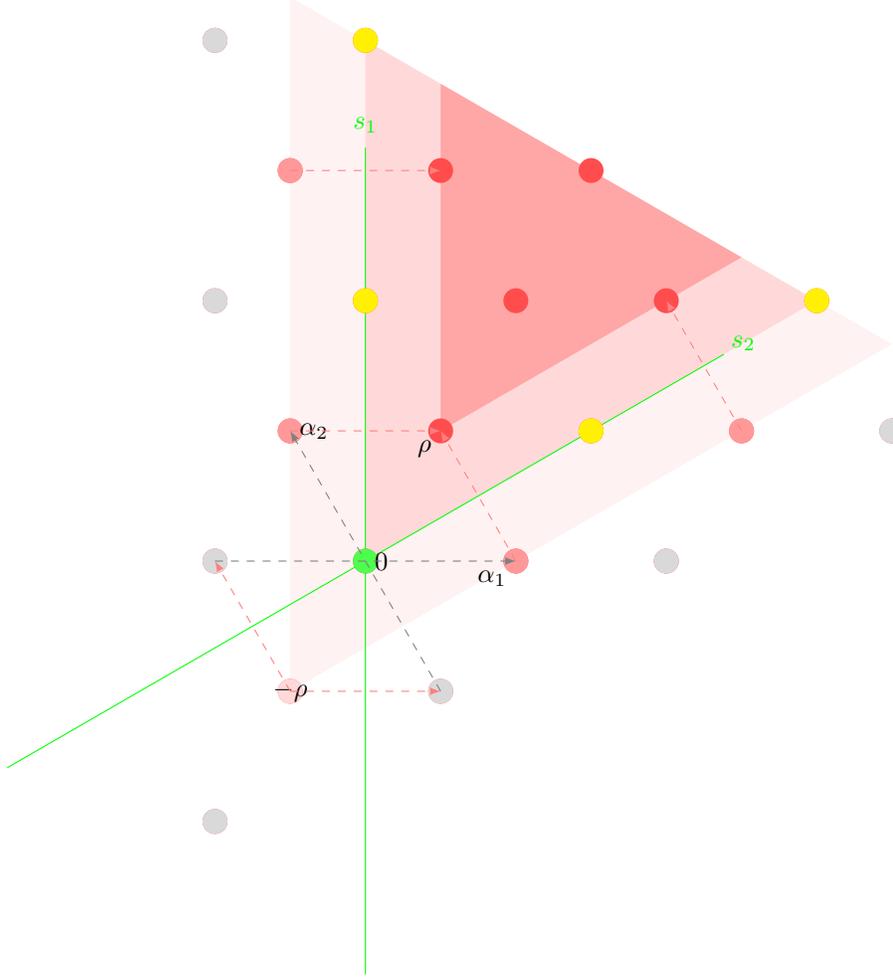
\begin{figure}

\centering
\begin{tikzpicture}[>=latex]

\pgfmathsetmacro\ax{2}
\pgfmathsetmacro\ay{0}
\pgfmathsetmacro\bx{2 * cos(120)}
\pgfmathsetmacro\by{2 * sin(120)}
\pgfmathsetmacro\lax{2*\ax/3 + \bx/3}
\pgfmathsetmacro\lay{2*\ay/3 + \by/3}
\pgfmathsetmacro\lbx{\ax/3 + 2*\bx/3}
\pgfmathsetmacro\lby{\ay/3 + 2*\by/3}

%%%%%chambers
%by=1.732. bx=-1
\fill[red!5] (\bx, -\by) -- (\bx+4*\by * \lby,-\by+4*\lby) -- (\bx,-\by+4*2*\lby) -- cycle;
\fill[red!15] (0,0) -- (30:3*2 * \lby) -- (0,3*2 * \lby) -- cycle;
\fill[red!35] (-\bx, \by) -- (-\bx+2*\by * \lby,\by+2*\lby) -- (-\bx,\by+2*2*\lby) -- cycle;
%%%%%%

%%%%axes
\foreach \k in {1,3,4,6} {
  \draw[green] (0,0) -- +(\k * 60 + 30:5.5);
}

\begin{scope}
%\clip (90:5) \foreach \k in {1,...,6} { -- ++([rotate=\k * 60 + 60]90:5) };
\clip (2.2*\bx, -2.2*\by) -- (2.2*\bx+5.5*\by * \lby,-2.2*\by+5.5*\lby) -- (2.2*\bx,-2.2*\by+5.5*2*\lby) -- cycle;
\foreach \na in {-4,...,4} {
  \foreach \nb in {-4,...,4} {
    %%%%%%SL(3) coweights (0 mod 3 PGL)
    \node[circle,fill=red!70] at (\na * \ax + \nb * \bx, \na * \ay + \nb * \by) {};
    %can change draw to fill if you want...
    %%%%%%1 mod 3 coweights for PGL(3)
    %\node[circle,draw=red] at (\lax + \na * \ax + \nb * \bx, \lay + \na * \ay + \nb * \by) {};
    %%%%%%2 mod 3 coweights
    %\node[circle,draw=red] at (\lbx + \na * \ax + \nb * \bx, \lby + \na * \ay + \nb * \by) {};
  }
}
\node[circle,fill=gray!30] at (2 * \ax + 0 * \bx, 2 * \ay + 0 * \by) {};
\node[circle,fill=gray!30] at (0 * \ax + 2 * \bx, 0 * \ay + 2 * \by) {};
\node[circle,fill=gray!30] at (0 * \ax + -1 * \bx, 0 * \ay + -1 * \by) {};
\node[circle,fill=gray!30] at (-1 * \ax + 0 * \bx, -1 * \ay + 0 * \by) {};
\node[circle,fill=gray!30] at (1 * \ax + 4 * \bx, 1 * \ay + 4 * \by) {};
\node[circle,fill=gray!30] at (4 * \ax + 1 * \bx, 4 * \ay + 1 * \by) {};
\node[circle,fill=gray!30] at (-2 * \ax + -2 * \bx, -2 * \ay + -2 * \by) {};
\node[circle,fill=green!70] at (0 * \ax + 0 * \bx, 0 * \ay + 0 * \by) {};
\node[circle,fill=yellow] at (2 * \ax + 1 * \bx, 2 * \ay + 1 * \by) {};
\node[circle,fill=yellow] at (1 * \ax + 2 * \bx, 1 * \ay + 2 * \by) {};
\node[circle,fill=yellow] at (4 * \ax + 2 * \bx, 4 * \ay + 2 * \by) {};
\node[circle,fill=yellow] at (2 * \ax + 4 * \bx, 2 * \ay + 4 * \by) {};
\node[circle,fill=red!40] at (1 * \ax + 0 * \bx, 1 * \ay + 0 * \by) {};
\node[circle,fill=red!40] at (3 * \ax + 1 * \bx, 3 * \ay + 1 * \by) {};
\node[circle,fill=red!40] at (0 * \ax + 1 * \bx, 0 * \ay + 1 * \by) {};
\node[circle,fill=red!40] at (1 * \ax + 3 * \bx, 1 * \ay + 3 * \by) {};
\node[circle,fill=red!15] at (-1 * \ax + -1 * \bx, -1 * \ay + -1 * \by) {};
\end{scope}
%%%%%%arrows and labels
%\draw[black,->] (0,0) -- (\ax,\ay) node[below left] {\(\alpha_1\)};
%\draw[black,->] (0,0) -- (\bx,\by) node[right] {\(\alpha_2\)};
\node[below left] at (0:2) {\(\alpha_1\)};
\node[right] at (120:2) {\(\alpha_2\)};
%\draw[->] (0,0) -- (\lax,\lay) node[below] {\(\lambda_1\)};
%\draw[->] (0,0) -- (\lbx,\lby) node[below right] {\(\lambda_2\)};
%\node at (0:5) {\(\rho_2\)};
%\node at (120:5) {\(\rho_1\)};
%\node at (180:5) {\(\rho_{12}\)};
%\node at (-120:5) {\(\rho_{121} = \rho_{212}\)};
%\node at (-60:5) {\(\rho_{21}\)};
\node[below left] at (60:2) {\(\rho\)};
\node at (240:2) {\(-\rho\)};
\node[right] at (0:0) {\(0\)};
\node[green] at (30:5.8) {\(s_2\)};
\node[green] at (90:5.8) {\(s_1\)};
\draw[red!50,dashed,->] (\ax,\ay) -- (\ax+\bx,\ay+\by) ;
\draw[red!50,dashed,->] (3*\ax+\bx,3*\ay+\by) -- (3*\ax+2*\bx,3*\ay+2*\by) ;
\draw[red!50,dashed,->] (-\ax-\bx,-\ay-\by) -- (-\ax,-\ay) ;
\draw[red!50,dashed,->] (\bx,\by) -- (\ax+\bx,\ay+\by) ;
\draw[red!50,dashed,->] (\ax+3*\bx,\ay+3*\by) -- (2*\ax+3*\bx,2*\ay+3*\by) ;
\draw[red!50,dashed,->] (-\ax-\bx,-\ay-\by) -- (-\bx,-\by) ;
\draw[gray,dashed,->] (-\ax,-\ay) -- (\ax,\ay) ;
\draw[gray,dashed,->] (-\bx,-\by) -- (\bx,\by) ;

\end{tikzpicture}
\caption{ Lattice of coweights of $\mathrm{SL}(3)$; depicts the structure of $\widetilde{C}$ as a $(\mathcal{H}^{fin})^{\otimes S}$ module. The module is generated by the shifted dominant cone $-\rho +\Lambda_+$. The generator $0$ satisfies the reflection relation. Generators (colored yellow) along a wall satisfy the reflection relations for only one simple root. Dashed red arrow indicated generators are related as $\mathrm{Eis}_{-1}$ and $\mathrm{Eis}_1$ (see $\mathrm{PGL}(2)$ example).}\label{sl3fig}

\end{figure}

%%%root pacakges and examples
%https://tex.stackexchange.com/questions/30301/root-systems-and-weight-lattices-with-pstricks

%https://en.wikibooks.org/wiki/LaTeX/Colors

\begin{enumerate}

\item $\lambda=0$ (Principal Orbit)
\begin{equation} \label{principalorbit}
\mathrm{Avg}_{i}^{\{0,1\}}J_0=\mathrm{Avg}_{i}^{\{0,\infty\}}J_0=\mathrm{Avg}_{i}^{\{1,\infty\}}J_0 \ \text{for} \ i \in \{1,2\}
\end{equation}

\item $\lambda \in W \cdot \rho$
%$\lambda$ in $W$-orbit of $\rho$

\begin{equation}\label{rhoorbit1}
\mathrm{Avg}_1^s J_{\alpha_2}=\mathrm{Avg}_1^s T_{1}^{S \setminus \{s\}} J_{\rho} \ \text{for} \ s \in S
\end{equation}

\begin{equation}\label{rhoorbit2}\mathrm{Avg}_2^s J_{\alpha_1}=\mathrm{Avg}_2^s T_{2}^{S \setminus \{s\}} J_{\rho} \ \text{for} \ s \in S
\end{equation}

\begin{equation}\label{rhoorbit3}\mathrm{Avg}_i^sJ_{-\rho} \in \mathrm{Span}_{(\mathcal{H}^{fin})^{\otimes S}} \{J_0,J_{\alpha_1},J_{\alpha_2},J_{\rho}\} \ \text{for} \ i \in \{1,2\}, \ s \in S 
\end{equation}

\item $\langle \check{\alpha_i}, \lambda \rangle =0$, $\lambda \neq 0$ (walls of dominant cone)

\begin{equation}\label{wall orbit}\mathrm{Avg}_{i}^{\{0,1\}}J_{\lambda}=\mathrm{Avg}_{i}^{\{0,\infty\}}J_{\lambda}=\mathrm{Avg}_{i}^{\{1,\infty\}}J_{\lambda}
\end{equation}

\item  $\langle \check{\alpha_i}, \lambda \rangle =-1$, $\lambda \neq -\rho$ (walls of $-\rho$ shifted dominant cone)

\begin{equation}\label{shiftedwallorbit}\mathrm{Avg}_i^s J_{\lambda}=\mathrm{Avg}_i^s T_{i}^{S \setminus \{s\}} J_{s_i \cdot \lambda} \ \text{for} \ s \in S
\end{equation}

\end{enumerate}

%\noindent Frog

\noindent We want to show that $\widetilde{C} \rightarrow C_{Eis}$ given by $1 \mapsto \mathrm{Eis}_0$ is an isomorphism. We study the map over map $(\mathcal{H}^{fin})^{\otimes S}$. For $\lambda \in \Lambda_+$, define the the $(\mathcal{H}^{fin})^{\otimes S}$ submodules $\widetilde{C}^{\lambda} \subset \widetilde{C}$ and $C^{\lambda}_{Eis} \subset C_{Eis}$ as follows.

$$\widetilde{C}^{\lambda}:= \mathrm{Span}_{(\mathcal{H}^{fin})^{\otimes S}} \{J_{w \cdot \lambda}: w \in W,\  w\cdot\lambda \in -\rho+\Lambda_+\}$$

$$C_{Eis}^{\lambda}:= \mathrm{Span}_{(\mathcal{H}^{fin})^{\otimes S}} \{\mathrm{Eis}_{w \cdot \lambda}: w \in W,w\cdot \  \lambda \in -\rho+\Lambda_+\}$$

\noindent It suffices to show that 

\begin{equation}\label{directsum}
    C_{Eis}\cong \oplus_{\lambda \in \Lambda_+} C_{Eis}^{\lambda}
\end{equation}
\begin{equation}\label{dimension}
  \mathrm{dim}_{\mathbb{C}}(C_{Eis}^{\lambda}) \geq \mathrm{dim}_{\mathbb{C}}(\widetilde{C}^{\lambda}) \ \text{for} \ \lambda \in \Lambda_+  
\end{equation}

\noindent These are established by Propositions \ref{sl3dimprop} and \ref{sl3algdimprop}.

\begin{prop}\label{sl3dimprop}
    Equation \ref{directsum} is true and 

    $$\mathrm{dim}_{\mathbb{C}}(C_{Eis}^{\lambda})=
    \begin{cases} 
      69 & \lambda=0 \\
      6^3+3^3+3^3+1^3 & \lambda=\rho \\
      3^3 \cdot 5 & \langle \check{\alpha_i}, \lambda \rangle =0, \ \lambda \neq 0  \\
      3^3\cdot(2^3+1) & \langle \check{\alpha_i}, \lambda \rangle =1, \ \lambda \neq \rho \\
      6^3 & \lambda \in 2\rho+\Lambda_+
   \end{cases}
   $$
\end{prop}

\begin{prop}\label{sl3algdimprop}

 $$\mathrm{dim}_{\mathbb{C}}(\widetilde{C}^{\lambda}) \leq 
    \begin{cases} 
      69 & \lambda=0 \\
      6^3+3^3+3^3+1^3 & \lambda=\rho \\
      3^3 \cdot 5 & \langle \check{\alpha_i}, \lambda \rangle =0, \ \lambda \neq 0  \\
      3^3\cdot(2^3+1) & \langle \check{\alpha_i}, \lambda \rangle =1, \ \lambda \neq \rho \\
      6^3 & \lambda \in 2\rho+\Lambda_+
   \end{cases}
   $$

\end{prop}

\subsection{Proof of Proposition \ref{sl3dimprop}}
   To prove Proposition \ref{sl3dimprop}, we first describe the geometry of the fibers $\mathrm{Bun}_G(\mathbb{P}^1,S) \rightarrow \mathrm{Bun}_G(\mathbb{P}^1)$ and identify the Eisenstein objects $\mathrm{Eis}_{\lambda}$ for $\lambda \in -\rho+\Lambda_+$. For $\lambda \in \Lambda_+$, let $\mathrm{Bun}_{G}^{\lambda}(\mathbb{P}^1,S)$ denote the fiber above $\lambda \in \Lambda_+$. We will find that except for $\mathrm{Eis}_{-\rho}$, all these Eisenstein objects are nonzero only on a single point, which lies in $\mathrm{Bun}_G^{\tilde{\lambda}}(\mathbb{P}^1,S)$, where $\tilde{\lambda} \in \Lambda_+$ is in the $W$ orbit of $\lambda$. We will also find that for $\lambda \in \Lambda_+ \setminus \{0,\rho\}$, $C_{Eis}^{\lambda}$ is equal to the space of all automorphic functions taking nonzero values only on points of $\mathrm{Bun}_G^{\lambda}(\mathbb{P}^1,S)$.

   Objects of $\mathrm{Bun}_G(\mathbb{P}^1)$ are represented by rank $3$ vector bundles $\mathcal{E}$, whose determinant bundle is trivial. Objects of $\mathrm{Bun}_G(\mathbb{P}^1,S)$ are represented by $\mathcal{E} \in \mathrm{Bun}_G(\mathbb{P}^1)$ with flags $F_s=(\ell_s,p_s)$, $\ell_s \subset p_s \subset \mathcal{E}|_s$.

   \subsubsection{  $\mathcal{E} \cong \mathcal{O}(0)$}

$\mathrm{Bun}_G^0(\mathbb{P}^1,S)$ is identified with the orbits of the triple flag variety $G  \backslash \mathcal{B}^S$. The generic configuration is when the flags are pairwise transverse and the following two conditions are satisfied:

\begin{itemize}
    \item The lines $\ell_s$ are not coplanar.
    \item The planes $p_s$ are not concurrent.
\end{itemize}

\noindent For $(p,q) \in S \times S$ with $p \neq q$, there is a map $\pi_{p,q}:G \backslash \mathcal{B}^S \rightarrow G \backslash (\mathcal{B} \times \mathcal{B})$. Identify the points of $G \backslash (\mathcal{B} \times \mathcal{B}) \cong B \backslash G/B$ with $ W$ by relative position of flags. Explicitly, for $w \in W$, $(F_1,F_2)$, is in relative position $w$,

\begin{itemize}
\item $w=1$ if $\ell_1=\ell_2$ and $p_1=p_2$
\item $w=s_1$ if $\ell_1 \neq \ell_2$ and $p_1=p_2$
\item $w=s_2$ if $\ell_1 =\ell_2$ and $p_1 \neq p_2$
\item $w=s_2s_1$ if $\ell_2 \in p_1$, $\ell_1 \notin p_2$
\item $w=s_1s_2$ if $\ell_2 \notin p_1$, $\ell_1 \in p_2$
\item $w=s_3$ if $\ell_2 \notin p_1$, $\ell_1 \notin p_2$
    
\end{itemize}

\noindent If $(F_0,F_1)$ are in relative position $w$ and $(F_1,F_{\infty})$ are in relative position $w'$, then the possible relative positions of $(F_0,F_{\infty})$ are exactly those $w'' \in W$ such that $T_{w''}$ has a nonzero coefficient in $T_{w'}T_w \in \mathcal{H}^{fin}$. In particular, if $\ell(w)+\ell(w')=\ell(w' w)$, then the relative position of $(F_0,F_{\infty})$ must be $w'w$. The other cases are

\begin{itemize}
\item $w=w'=s_i$. $T_{w'}T_w=(q-1)T_{s_i}+q$.
\item $w=s_i$, $w'=s_js_i$. $T_{w'}T_w=(q-1)T_{s_js_i}+qT_{s_j}$.
\item $w=s_i$, $w'=s_3$. $T_{w'}T_w=(q-1)T_{s_3}+qT_{s_is_j}$ .

\item $w=s_is_j$, $w'=s_i$. $T_{w'}T_w=(q-1)T_{s_is_j}+qT_{s_j}$.
\item $w=s_is_j$, $w'=s_js_i$. $T_{w'}T_w=(q-1)T_{s_3}+q(q-1)T_{s_j}+q^2$
\item $w=w'=s_is_j$. $T_{w'}T_w=(q-1)T_{s_3}+qT_{s_js_i}$
\item $w=s_is_j$, $w'=s_3$. $T_{w'}T_w=(q-1)^2T_{s_3}+q(q-1)T_{s_js_i}+q(q-1)T_{s_is_j}+q^2T_{s_i}$.
\item $w=s_3$, $w'=s_i$. $T_{w'}T_w=(q-1)T_{s_3}+qT_{s_js_i}$.
\item $w=s_3$, $w'=s_js_i$. $T_{w'}T_w=(q-1)^2T_{s_3}+q(q-1)T_{s_js_i}+q(q-1)T_{s_is_j}+q^2T_{s_i}$.
\item $w=s_3$, $w'=s_3$. $T_{w'}T_w=(q-1)(q^2-q+1)T_{s_3}+q(q-1)^2T_{s_is_j}+q(q-1)^2T_{s_js_i}+q^2(q-1)T_{s_i}+q^2(q-1)T_{s_j}+q^3$

\end{itemize}

\noindent $s_i$ is one simple reflection, $s_j$ is the other. Let $\pi:G \backslash (\mathcal{B}^S) \rightarrow (B \backslash G/B)^3$ be given by $(\pi_{0,1},\pi_{1,\infty},\pi_{0,\infty})$. Let $c_0(w,w',w'')$ be the preimage of $(w,w',w'')$. From the above calculation, we find that $c_0(w,w',w'')$ is nonempty exactly for the following triples:
\begin{enumerate}

\item $(w,w',w'')$, with $w''=w'w$. There are exactly $36$ such triples.
\item $(w,w',w'')$ is one of the following $33$ triples:

$$(s_1,s_1,s_1),(s_2,s_2,s_2),(s_1,s_2s_1,s_2s_1),(s_2,s_1s_2,s_1s_2),(s_1,s_3,s_3),(s_2,s_3,s_3),(s_1s_2,s_1,s_1s_2),(s_2s_1,s_2,s_2s_1),(s_1s_2,s_2s_1,s_3),$$ $$(s_1s_2,s_2s_1,s_1), (s_2s_1,s_1s_2,s_3),
(s_2s_1,s_1s_2,s_2), (s_1s_2,s_1s_2,s_3),(s_2s_1,s_2s_1,s_3),(s_1s_2,s_3,s_3),(s_1s_2,s_3,s_1s_2),(s_1s_2,s_3,s_2s_1),$$ $$(s_2s_1,s_3,s_3),(s_2s_1,s_3,s_2s_1), (s_2s_1,s_3,s_1s_2), (s_3,s_1,s_3),(s_3,s_2,s_3),(s_3,s_2s_1,s_3),(s_3,s_2s_1,s_2s_1),(s_3,s_2s_1,s_1s_2),$$ $$(s_3,s_1s_2,s_3),(s_3,s_1s_2,s_1s_2),(s_3,s_1s_2,s_2s_1),(s_3,s_3,s_1),(s_3,s_3,s_2),(s_3,s_3,s_1s_2),(s_3,s_3,s_2s_1),(s_3,s_3,s_3)$$

\end{enumerate}

\noindent One can check that each of these loci $c_0(w,w',w'')$  has exactly one isomorphism class of objects, except for the locus $c_0(s_3,s_3,s_3)$, classifying triples of pairwise transverse flags. This locus is as follows:

\begin{centering}

\begin{tikzcd}
   & c_0(s_3,s_3,s_3; \emptyset) \arrow[ld]  \arrow[rd] &    \\
  c_0(s_3,s_3,s_3; \{s_1\})  \arrow[rd]&  & c_0(s_3,s_3,s_3; \{s_2\} )  \arrow[ld]    \\
   & c_0(s_3,s_3,s_3; \{s_1,s_2\})  &  
\end{tikzcd}
%\arrow[r,"\pi_2"] \arrow[l,"\pi_1"] 

\end{centering}

\noindent $c_0(s_3,s_3,s_3;\{s_1,s_2\})$ is the configuration where $\ell_s$ are coplanar and $p_s$ are concurrent. $c_0(s_3,s_3,s_3;\{s_1\})$ is the configuration where $\ell_s$ are coplanar and $p_s$ are non concurrent. $c_0(s_3,s_3,s_3;\{s_2\})$ is the configuration where $\ell_s$ are not coplanar and $p_s$ are concurrent. $c_0(s_3,s_3,s_3;\emptyset)$ is the generic configuration. The loci of flags in generic configuration forms a space isomorphic to the complement of three points in $\mathbb{P}^1$; the remaining subloci of $c_0(s_3,s_3,s_3)$ have a unique point.

\begin{definition}
    $C_{loc}^0 \subset C_{Aut}$ is the subspace of functions supported on points where the bundle is trivial and constant along the generic locus $c_0(s_3,s_3,s_3;\emptyset)$.
\end{definition}

\noindent By enumerating points we have shown that $\mathrm{dim}_{\mathbb{C}}(C_{loc}^0)=72$. $C_{Eis}^0$ is generated over $(\mathcal{H}^{fin})^{\otimes S}$ by $\mathrm{Eis}_0=\underline{1}_{c_0(1,1,1)}$. Furthermore, all functions on $C_{Eis}^0$ are constant along the generic locus $c_0(s_3,s_3,s_3;\emptyset)$. Therefore, $C_{Eis}^0 \subset C_{loc}^0$.

\begin{lem}
    $C_{Eis}^0$ is codimension three in $C_{loc}^0$.
\end{lem}

\begin{proof}
We describe the equations cutting out $C_{Eis}^0$ in $C_{loc}^0$. Let $c_0(*)$ denote the locus where $\ell_s$ are not coplanar and $p_s$ are not concurrent. The following are subloci of $c_0(*)$ organized so that $x \rightarrow y$ means $y$ is contained in the closure of $x$.

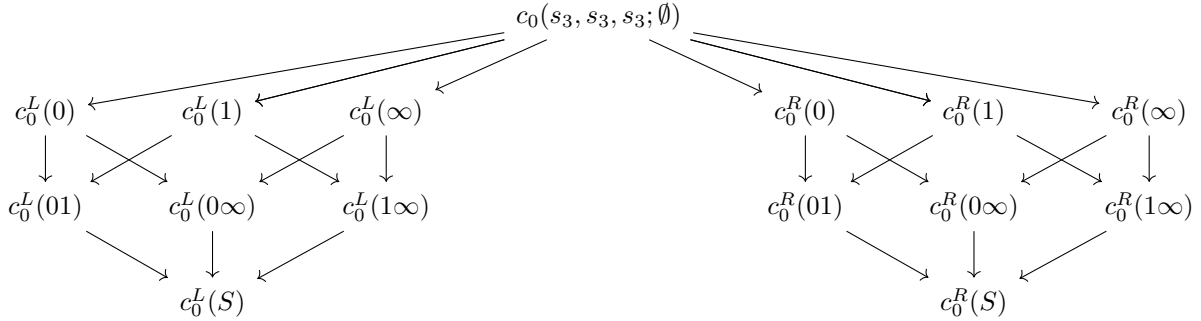
\begin{centering}

\begin{tikzcd}
    & & & c_0(s_3,s_3,s_3;\emptyset) \arrow[ld] \arrow[lld] \arrow[lld] \arrow[llld] \arrow[rd] \arrow[rrd] \arrow[rrd] \arrow[rrrd] & & & \\
  c_0^L(0) \arrow[d] \arrow[rd]   & c_0^L(1) \arrow[ld] \arrow[rd] & c_0^L(\infty) \arrow[ld] \arrow[d] & & c_0^R(0) \arrow[d] \arrow[rd] & c_0^R(1) \arrow[ld] \arrow[rd] & c_0^R(\infty) \arrow[ld] \arrow[d] \\
  c_0^L(01) \arrow[rd]   & c_0^L(0 \infty) \arrow[d] & c_0^L(1 \infty) \arrow[ld] & & c_0^R(0 1) \arrow[rd] & c_0^R(0 \infty) \arrow[d] & c_0^R(1\infty) \arrow[ld] \\
    & c_0^L(S) &  & &  & c_0^R(S) & 
\end{tikzcd}
    
\end{centering}
%\ell_0 \in \p_1, \ell_1 \notin p_{\infty}, \ell_{\infty} \notin p_0

\noindent For $R \subset S$, $c_0^L(R) \subset c_0(*)$ is the sublocus where $\ell_s \in p_{L(s)}$ if and only if $s \in R$, where $L(s)$ denotes the predecessor of $s$ in the cyclic ordering $0 \rightarrow 1 \rightarrow \infty \rightarrow 0$. $c_0(R) \subset c_0(*)$ is the sublocus where $\ell_s \in P_{R(s)}$ if and only if $s \in R$, where $R(s)$ denotes the successor of $s$ in the same cyclic ordering. In the previous notation

\begin{enumerate}
    \item $c_0^L(S)=c_0(s_2s_1,s_2s_1,s_1s_2)$
    \item $c^L(01)=c_0(s_2s_1,s_3,s_1s_2),c^L(0 \infty)=c_0(s_3,s_2s_1,s_1s_2),c_0^L(1 \infty)=c_0(s_2s_1,s_2s_1,s_3)$
    \item $c^L(0)=c_0(s_3,s_3,s_1s_2),c^L(1)=c_0(s_2s_1,s_3,s_3),c_0^L( \infty)=c_0(s_3,s_2s_1,s_3)$

    \item $c_0^R(S)=c_0(s_1s_2,s_1s_2,s_2s_1)$
    \item $c^R(01)=c_0(s_1s_2,s_1s_2,s_3),c^R(0 \infty)=c_0(s_1s_2,s_3,s_2s_1),c_0^R(1 \infty)=c_0(s_3,s_1s_2,s_2s_1)$
    \item $c^R(0)=c_0(s_1s_2,s_3,s_3),c^R(1)=c_0(s_3,s_1s_2,s_3),c_0^R( \infty)=c_0(s_3,s_3,s_2s_1)$

\end{enumerate}

$C_{Eis}^0 \subset C_{loc}^0$ is the subspace of functions, $f$, such that

$$f(c_0(s_3,s_3,s_3;\emptyset))-f(c_0(s_3,s_3,s_3;\{s_1\}))-f(c_0(s_3,s_3,s_3;\{s_2\}))+f(c_0(s_3,s_3,s_3;\{s_1,s_2\}))=0$$

$$f(c_0(s_3,s_3,s_3;\emptyset))+\sum_{R \subset S; R \neq \emptyset} (-1)^{|R|}f(c_0^L(R))=0$$
$$f(c_0(s_3,s_3,s_3;\emptyset))+\sum_{R \subset S; R \neq \emptyset} (-1)^{|R|}f(c_0^R(R))=0$$

\noindent $f(c_0(s_3,s_3,s_3;\emptyset))$ is the common value of $f$ on any point of the generic locus $c_0(s_3,s_3,s_3;\emptyset)$.

\end{proof}

%https://arxiv.org/pdf/math/0302151.pdf for next sections

\subsubsection{Interlude on Bundles with a Positive Splitting}

The following is a special case. The general principle will be elaborated upon in a future document. Suppose that $\mathcal{E} \cong \mathcal{O}(\lambda)$ admits a \textit{positive} splitting $\mathcal{E}\cong \mathcal{E}_1 \oplus \mathcal{E}_2$, which means that $\mathrm{Hom}(\mathcal{E}_1,\mathcal{E}_2 )=0$. For example, if $\mathcal{E}_1\cong \mathcal{O}(m) \oplus \mathcal{O}(n)$ and $\mathcal{E}_2 \cong \mathcal{O}(k)$ then the positivity condition is $m,n \geq k+1$. Let $P \supset B$ be the parabolic subgroup corresponding to the splitting. If $\mathcal{E}_1$ is rank two, then $P=P_{s_1}$ and if $\mathcal{E}_1$ is rank one, then $P=P_{s_2}$. The subbundle $\mathcal{E}_1$ is stable under $\mathrm{Aut}(\mathcal{E})$, so there is a subspace $\mathcal{E}^{\mathrm{stab}}_s \subset \mathcal{E}|_s$ given by restriction of $\mathcal{E}_1$. Let $\mathrm{fib}_s:\mathrm{Bun}_G^{\lambda}(\mathbb{P}^1,S) \rightarrow P \backslash G/B$ be given by relative position of $(\mathcal{E}^{\mathrm{stab}}_s,F_s)$. For example, if $\mathcal{E}_1$ is rank two, the relative position, $w$ is given by:

\begin{itemize}
    \item $w=1$ if $p_s=\mathcal{E}^{\mathrm{stab}}_s$
    \item $w=s_2$ if $\ell_s \subset \mathcal{E}^{\mathrm{stab}}_s$ but $p_s \neq \mathcal{E}|_s$
    \item $w=s_1s_2$ if $\ell \notin \mathcal{E}^{\mathrm{stab}}_s$
\end{itemize}

\noindent There is also a map $\mathrm{Bun}_G^{\lambda}(\mathbb{P}^1) \rightarrow \mathrm{Bun}_L^{\lambda_1}(\mathbb{P}^1)$, given by $\mathcal{E} \mapsto \mathcal{E}_1 \oplus \mathcal{E}/\mathcal{E}_1$, where $L \subset P$ is the Levi subgroup and $\mathcal{E}_1 \cong \mathcal{O}(\lambda_1)$. At the level of rational points, this can be lifted to included parabolic structure:

$$\mathrm{split}:\mathrm{Bun}_G^{\lambda}(\mathbb{P}^1,S) \rightarrow \mathrm{Bun}_L^{\lambda_1}(\mathbb{P}^1,S)$$

\noindent For example, if $\mathcal{E}_1$ is rank two, the parabolic structure for $\mathcal{E}_1$ at $s$ is given by $p_s \cap \mathcal{E}_s^{\mathrm{stab}}$ if $p_s$ is transverse to $\mathcal{E}_s^{\mathrm{stab}}$ and otherwise by $\ell_s$. The splitting map is not continuous on the underlying moduli spaces.

Suppose further that the splitting $\mathcal{E} \cong \mathcal{E}_1 \oplus \mathcal{E}_2$ is \textit{very positive}, which means $\mathrm{Hom}(\mathcal{E}_1,\mathcal{E}_2\otimes \omega_{\mathbb{P}^1}(S))=0$. For example, if $\mathcal{E}_1\cong \mathcal{O}(m) \oplus \mathcal{O}(n)$ and $\mathcal{E}_2 \cong \mathcal{O}(k)$ then the condition is $m,n \geq k+2$. Calculating the action of $\mathrm{Aut}(\mathcal{E})$ on $\prod_{s \in S} \mathcal{E}|_s$ shows that the product of the splitting map and $\mathrm{fib}:=\prod_{s \in S} \mathrm{fib}_s$ is a bijection on points.

$$(P \backslash G /B)^S  \leftarrow \mathrm{Bun}_G^{\lambda}(\mathbb{P}^1,S) \rightarrow \mathrm{Bun}_L^{\lambda_1}(\mathbb{P}^1,S)$$

\subsubsection{$\mathcal{E} \cong \mathcal{O}(\rho)$}

The $B$-bundle $\mathcal{O}(1) \subset \mathcal{O}(1) \oplus \mathcal{O} \subset \mathcal{E}$ is stable under $\mathrm{Aut}(\mathcal{E})$. For $s \in S$, there is a flag $F_s^{\mathrm{stab}}=(\ell_s^{\mathrm{stab}},p_s^{\mathrm{stab}}) \subset \mathcal{E}|_s$, given by restriction of the stable $B$-bundle, also invariant under $\mathrm{Aut}(\mathcal{E})$. There is a map $\mathrm{fib}_s: \mathrm{Bun}_G^{\rho}(\mathbb{P}^1,S) \rightarrow B \backslash G/B$ given by the relative position $(F_s^{\mathrm{stab}}, F_s)$ of the flag $F_s$ in the fiber at $s$ to the stable $B$-bundle. Let $c_{\rho}(w_0,w_1,w_{\infty})$ denote the locus where the relative position of $F_s$ to the stable flag is $w_s \in W$.

There are two splitting maps, for $i=1,2$:
$$\mathrm{split}_i:\mathrm{Bun}_G^{\mathbb{\rho}}(\mathbb{P}^1,S) \rightarrow \mathrm{Bun}_{\mathrm{PGL}(2)}^1(\mathbb{P}^1,S)$$

\noindent The parabolic structure at $s \in S$ for $\mathrm{split}_1$ is given by the distinguished line $\ell_s^{\mathrm{dist},1} \subset p_s^{\mathrm{stab}}$ defined as $\ell_s^{\mathrm{dist},1}=p_s \cap p_s^{\mathrm{stab}}$ if $F_s$ is transverse to $p_s^{\mathrm{stab}}$ and $\ell_s$ otherwise. The parabolic structure for $\mathrm{split}_2$ is given by the distinguished plane $\ell_s^{\mathrm{dist},2} \subset \mathcal{E}|_s/\ell_{s}^{\mathrm{stab}}$ given by $(\ell_s \oplus \ell_s^{\mathrm{stab}})/\ell_s^{\mathrm{stab}}$ if $F_s$ is transverse to $\ell_s^{\mathrm{stab}}$ and $p_s/\ell_{s}^{\mathrm{stab}}$, otherwise.

Explicitly, the points of $c_{\rho}(w_0,w_1,w_{\infty})$ are as follows.
\begin{enumerate}
\item If for each $i=1,2$, there is at least one $s \in S$ such that  $\ell(w_s s_1)>\ell(w_s)$, then the locus consists of a single point.
\item $\ell(w_s s_1)<\ell(w_s)$ for all $s \in S$, but there is at least one $s' \in S$ such that $\ell(w_{s'} s_2) > \ell(w_{s'})$. This locus consists of two points. The generic configuration, $c_{\rho}(w_0,w_1,w_{\infty};\emptyset)$ is where the distinguished lines $\ell_{s}^{\mathrm{dist},1}$ are not contained in the image of a map $\mathcal{O} \rightarrow \mathcal{O}(1) \oplus \mathcal{O}$. The degenerate locus, $c_{\rho}(w_0,w_1,w_{\infty};\{s_1\})$ is where there is such a map.

\item $\ell(w_s s_2)<\ell(w_s)$ for all $s \in S$, but there is at least one $s' \in S$ such that $\ell(w_{s'} s_1) > \ell(w_{s'})$. The generic configuration, $c_{\rho}(w_0,w_1,w_{\infty};\emptyset)$ is where the distinguished lines $\ell_{s}^{\mathrm{dist},2}$ are not contained in the image of a map $\mathcal{O}(-1) \rightarrow \mathcal{E}/\mathcal{O}(1)$. The degenerate locus, $c_{\rho}(w_0,w_1,w_{\infty};\{s_2\})$ is where there is such a map.

    \item $w_0=w_1=w_{\infty}=s_3$. This locus has four points

    \begin{centering}

\begin{tikzcd}
   & c_{\rho}(s_3,s_3,s_3; \emptyset) \arrow[ld]  \arrow[rd] &    \\
  c_{\rho}(s_3,s_3,s_3; \{s_1\})  \arrow[rd]&  & c_{\rho}(s_3,s_3,s_3; \{s_2\} )  \arrow[ld]    \\
   & c_{\rho}(s_3,s_3,s_3; \{s_1,s_2\})  &  
\end{tikzcd}
%\arrow[r,"\pi_2"] \arrow[l,"\pi_1"] 

\end{centering}

\noindent For $\delta \subset \{s_1,s_2\}$ $c_{\rho}(s_3,s_3,s_3;\delta)$ is the locus where the distinguished lines $\ell_s^{\mathrm{dist},1}=p_s \cap p_s^{\mathrm{stab}}$ are contained in the image of a map $\mathcal{O} \rightarrow \mathcal{O}(1)\oplus \mathcal{O}$ if and only if $s_1 \in \delta$ and the distinguished lines $\ell_s^{\mathrm{dist},2}=p_s/\ell_{s}^{\mathrm{stab}}$ is contained in the image of a map $\mathcal{O}(-1) \rightarrow \mathcal{E}/\mathcal{O}(1)$ if and only if $s_2 \in \delta$. 
\end{enumerate}

The Eisenstein objects in $C_{Eis}^{\rho}$ are

$$\mathrm{Eis}_{\rho}=\underline{1}_{c_{\rho}(1,1,1)}$$

$$\mathrm{Eis}_{s_1 \cdot \rho}=\underline{1}_{c_{\rho}(s_1,s_1,s_1;\{s_1\})}$$

$$\mathrm{Eis}_{s_2 \cdot \rho}=\underline{1}_{c_{\rho}(s_2,s_2,s_2;\{s_2\})}$$

$$\mathrm{Eis}_{-\rho}=\underline{1}_{c_{\rho}(s_3,s_3,s_3;\{s_1,s_2\})}+\underline{1}_{c_0(s_3,s_3,s_3;\emptyset)}$$

[Check the last calculation] 

The finite Hecke module generated by $\mathrm{Eis}_{\rho}$ consists of all functions on the points $\mathrm{Bun}_{G}^{\rho}(\mathbb{P}^1,S)$ constant along the loci $c_{\rho}(w_0,w_1,w_{\infty})$. $\mathrm{Eis}_{s_i \cdot \rho}$ generates, under finite Hecke modification, the constant function on  points $c_{\rho}(w_0,w_1,w_{\infty};\{s_i\})$ for $(w_0,w_1,w_{\infty}) \neq (s_3,s_3,s_3)$ as well as the function $$\underline{1}_{c_{\rho}(s_3,s_3,s_3;\{s_1,s_2\})}+\underline{1}_{c_{\rho}(s_3,s_3,s_3;\{s_i\})}.$$

\noindent Therefore, $C_{Eis}^{\rho}$ consists of functions, $f$, vanishing away from the points of the loci $\mathrm{Bun}_G^{\rho}(\mathbb{P}^1,S)$ and $c_0(s_3,s_3,s_3;\emptyset)$ that are constant along $c_0(s_3,s_3,s_3;\emptyset)$ and satisfy

$$f(c_0(s_3,s_3,s_3;\emptyset))=f(c_{\rho}(s_3,s_3,s_3;\{s_1,s_2\}).$$

\noindent It follows that $$\mathrm{dim}_{\mathbb{C}}(C_{Eis}^{\rho})=\left| \mathrm{Bun}_G^{\rho}(\mathbb{P}^1,S) \right|=6^3+3^3+3^3+1^3.$$

   \subsubsection{ $\mathcal{E} \cong \mathcal{O}(\lambda)$, $\langle \check{\alpha_i}, \lambda \rangle =0, \ \lambda \neq 0 $}

   Without loss of generality, assume $\langle \check{\alpha_1},\lambda \rangle =0$. Then $\mathcal{E} \cong \mathcal{O}(k) \oplus \mathcal{O}(k) \oplus \mathcal{O}(-2k)$, for some $k \geq 1$. There is a bijection of points

   $$\mathrm{Bun}_G^{\lambda}( \mathbb{P}^1,S) \leftrightarrow (P_{s_1} \backslash G/B)^S \times \mathrm{Bun}_{\mathrm{PGL(2)}}^0(\mathbb{P}^1,S) $$

\noindent $\mathrm{Eis}_{\lambda}$ is the constant function on the point corresponding to $c_0(S) \times (1,1,1)$. By the $\mathrm{PGL}(2)$ calculation, for every point, $\mathrm{pt}$, of $\mathrm{Bun}_{\mathrm{PGL(2)}}^0(\mathbb{P}^1,S)$, $C_{Eis}^{\lambda}$ contains the constant function on the point corresponding to $\mathrm{pt} \times (1,1,1)$. Furthermore, for $w_0,w_1,w_{\infty} \in \{1,s_2,s_1s_2\} \cong P_{s_1} \backslash G /B$, $$T_{w_0}^0T_{w_1}^1T_{w_{\infty}}^{\infty} \underline{1}_{\mathrm{pt} \times (1,1,1)}=\underline{1}_{\mathrm{pt} \times (w_0,w_1,w_{\infty})}.$$

\noindent Therefore, $C_{\mathrm{Eis}}^{\lambda}$ consists of all functions taking nonzero value only on the points of $\mathrm{Bun}_G^{\lambda}(\mathbb{P}^1,S)$, so

$$\mathrm{dim}_{\mathbb{C}}(C_{Eis}^{\lambda})=\left| (P_{s_1}\backslash G/B)^S \times \mathrm{Bun}_{\mathrm{PGL(2)}}^0(\mathbb{P}^1,S) \right|=3^3 \cdot 5$$

    \subsubsection{ $\mathcal{E} \cong \mathcal{O}(\lambda)$, $\langle \check{\alpha_i}, \lambda \rangle =1, \ \lambda \neq \rho $}    Without loss of generality, assume $\langle \check{\alpha_1},\lambda \rangle =1$. Then $\mathcal{E} \cong \mathcal{O}(k+1) \oplus \mathcal{O}(k) \oplus \mathcal{O}(-2k-1)$, for some $k \geq 1$. There is a bijection of points

   $$\mathrm{Bun}_G^{\lambda}( \mathbb{P}^1,S) \leftrightarrow (P_{s_1} \backslash G/B)^S \times \mathrm{Bun}_{\mathrm{PGL(2)}}^1(\mathbb{P}^1,S) $$

\noindent $\mathrm{Eis}_{\lambda}$ is the constant function on the point corresponding to $c_1(S) \times (1,1,1)$ and $\mathrm{Eis}_{s_1(\lambda)}$ is the constant function on the point corresponding to $c_{1}(\emptyset) \times (1,1,1)$. By the $\mathrm{PGL}(2)$ calculation, for every point, $\mathrm{pt}$, of $\mathrm{Bun}_{\mathrm{PGL(2)}}^1(\mathbb{P}^1,S)$, $C_{Eis}^{\lambda}$ contains the constant function on the point corresponding to $\mathrm{pt} \times (1,1,1)$. Furthermore, for $w_0,w_1,w_{\infty} \in \{1,s_2,s_1s_2\} \cong P_{s_1} \backslash G /B$, $$T_{w_0}^0T_{w_1}^1T_{w_{\infty}}^{\infty} \underline{1}_{\mathrm{pt} \times (1,1,1)}=\underline{1}_{\mathrm{pt} \times (w_0,w_1,w_{\infty})}.$$

\noindent Therefore, $C_{\mathrm{Eis}}^{\lambda}$ consists of all functions taking nonzero value only on the points of $\mathrm{Bun}_G^{\lambda}(\mathbb{P}^1,S)$, so

$$\mathrm{dim}_{\mathbb{C}}(C_{Eis}^{\lambda})=\left| (P_{s_1}\backslash G/B)^S \times \mathrm{Bun}_{\mathrm{PGL(2)}}^1(\mathbb{P}^1,S) \right|=3^3 \cdot (2^3+1)$$

   \subsubsection{ $\mathcal{E} \cong \mathcal{O}(\lambda)$, $\lambda \in 2\rho +\Lambda_+$}

There is  $B$-bundle stable under $\mathrm{Aut}(\mathcal{E})$. There is a map $\mathrm{fib}_s: \mathrm{Bun}_G^{\lambda}(\mathbb{P}^1,S) \rightarrow B \backslash G/B$ given by the relative position $(F_s^{\mathrm{stab}}, F_s)$ of the flag $F_s$ in the fiber at $s$ to the stable $B$-bundle. The points of the locus $\mathrm{Bun}_G^{\lambda}(\mathbb{P}^1,S)$ are identified with $(B \backslash G/B)^S$. Moreover, $\mathrm{Eis}_{\lambda}$ is identified with $\underline{1}_{(1,1,1)}$ and $T_{w_0}^0T_{w_1}^1T_{w_{\infty}}^{\infty}\mathrm{Eis}_{\lambda}$ is identified with $\underline{1}_{(w_0,w_1,w_{\infty})}$. Therefore, there is an isomorphism of $(\mathcal{H}^{fin})^{\otimes S}$ modules $(\mathcal{H}^{fin})^{\otimes S} \rightarrow C_{Eis}^{\lambda}$ given by $1 \mapsto \mathrm{Eis}_{\lambda}$. $\mathrm{dim}_{\mathbb{C}}(C_{Eis}^{\lambda})=|W|^3$.

\subsubsection{Proof of Equation \ref{directsum}}

First, check that $C_{Eis}^0 \cap C_{Eis}^{\rho}=0$. Indeed, every function in $C_{Eis}^0$ takes nonzero values only on points of $\mathrm{Bun}_G^0(\mathbb{P}^1,S)$, but every nontrivial function in $C_{Eis}^{\rho}$ takes nonzero value on some point of $\mathrm{Bun}_G^{\rho}(\mathbb{P}^1,S)$. Then, observe that the spaces $\{C_{Eis}^{\lambda}\}_{\lambda \in \Lambda_+ \setminus \{0,\rho\}} \cup \{C_{Eis}^0 \oplus C_{Eis}^{\rho}\}$ are pairwise orthogonal. This is because functions in $C_{Eis}^0 \oplus C_{Eis}^{\rho}$ take nonzero values only on points of $\mathrm{Bun}_G^0(\mathbb{P}^1,S) \cup \mathrm{Bun}_G^{\rho}(\mathbb{P}^1,S)$, whereas functions in $C_{Eis}^{\lambda}$ for $\lambda \in \Lambda_+ \setminus \{0,\rho\}$ only take nonzero value on points of $\mathrm{Bun}_G^{\lambda}(\mathbb{P}^1,S)$.

\begin{rmk}\label{sl3cusp}
    The space of cusp forms $C_{cusp} \subset C_{Aut}$ is the space orthogonal to $C_{Eis}$. We see that cusp forms are functions taking nonzero values only on the generic locus of $c_0(s_3,s_3,s_3;\emptyset)$, as well as the following points of $\mathrm{Bun}_G^0(\mathbb{P}^1,S) \cup \mathrm{Bun}_G^{\rho}(\mathbb{P}^1,S)$:

    \begin{enumerate}
        \item $c_0(s_3,s_3,s_3;\delta)$ for $\delta \subset \{s_1,s_2\}$ nonempty
        \item $c_0^L(R)$ for $R \subset S$ nonempty
        \item $c_0^R(R)$ for $R \subset S$ nonempty
        \item $c_{\rho}(s_3,s_3,s_3;\{s_1,s_2\})$ 
    \end{enumerate}

\noindent The space of cusp forms is given by the following equations.

$$f(c_{\rho}(s_3,s_3,s_3;\{s_1,s_2\})=-\sum_{\mathrm{pt} \in c_0(s_3,s_3,s_3;\{s_1,s_3\})} f (\mathrm{pt})$$

$$f(c_0^L(S))=-(q-1)f(c_0^L(0))=-(q-1)f(c_0^L(1))=-(q-1)f(c_0^L(\infty))=(q-1)^2f(c_0^L(01))=(q-1)^2f(c_0^L(0 \infty))$$ $$=(q-1)^2f(c_0^L(1 \infty))$$

$$f(c_0^R(S))=-(q-1)f(c_0^R(0))=-(q-1)f(c_0^R(1))=-(q-1)f(c_0^R(\infty))=(q-1)^2f(c_0^R(01))=(q-1)^2f(c_0^R(0 \infty))$$ $$=(q-1)^2f(c_0^R(1 \infty))$$

$$f(c_0(s_3,s_3,s_3;\{s_1,s_2\})=-(q-1)f(c_0(s_3,s_3,s_3;\{s_1\})=-(q-1)f(c_0(s_3,s_3,s_3;\{s_2\})$$

$$\sum_{\mathrm{pt} \in c_0(s_3,s_3,s_3;\{s_1,s_3\})} f (\mathrm{pt})+f(c_0(s_3,s_3,s_3;\{s_1\})+f(c_0^L(01))+f(c_0^R(01))=0$$

%\noindent Note that the motivic measure assigns unit measure to each point of $c_0(s_3,s_3,s_3;\{s_1,s_3\})$, so the integrals are simply sums of the value of the function over the points of the locus.

\noindent Counting points and constraints shows $\mathrm{dim}_{\mathbb{C}}(C_{cusp})=q$.
    
\end{rmk}

\subsection{Proof of Proposition \ref{sl3algdimprop}}

\subsubsection{$\lambda=0$}

$\widetilde{C}^{0}$ is generated over $(\mathcal{H}^{fin})^{\otimes S}$ by $J_{0}$. Using Equation \ref{principalorbit} we can always write any monomial $T_{w_0}^0T_{w_1}^1T_{w_{\infty}}^{\infty}$, $w_s \in W$, as a sum of monomials where for any $i \in \{1,2\}$
$$\ell(w_{\infty} s_i) < \ell(w_{\infty}) \implies \ell(w_{0} s_i) > \ell(w_{0}),\ \ell(w_{1} s_i) > \ell(w_{1}). $$

\noindent Let us list the triples $(w_0,w_1,w_{\infty})$ that satisfy this condition.

\begin{enumerate}
    \item $(w_0,w_1,1), \ w_0,w_1 \in W$
\item $(w_0,w_1,w_{\infty}), \ w_{\infty} \in \{s_1,s_2s_1\}, \ w_0,w_1 \in \{1,s_2,s_1s_2\} $
\item $(w_0,w_1,w_{\infty}), \ w_{\infty} \in \{s_2,s_1s_s\}, \ w_0,w_1 \in \{1,s_1,s_2s_1\} $
\item $(1,1,s_3)$
\end{enumerate}

\noindent There are $|W|^2+2 \cdot 3^2+2 \cdot 3^2+1=73$ such triples. Let $M$ be the set of 69 monomials formed from excluding the following four from the 73 listed monomials:

$$T_{s_1}^0T_{s_1}^1T_{s_1s_2}^{\infty},T_{s_2s_1}^0T_{s_1}^1T_{s_1s_2}^{\infty},T_{s_1}^0T_{s_2s_1}^1T_{s_1s_2}^{\infty},T_{s_2s_1}^0T_{s_2s_1}^1T_{s_1s_2}^{\infty}$$

\noindent  $\widetilde{C}^0$ is spanned over $\mathbb{C}$ by $M$. This follows from two Lemmas.

\begin{lem}
        $T_{s_1}^0T_{s_1}^1T_{s_1s_2}^{\infty} \in \mathrm{Span}_{\mathbb{C}}(M)$
\end{lem}

\begin{proof}

Explicitly,

\begin{equation}\label{miracle0}
T_{s_1}^0T_{s_1}^1T_{s_1s_2}^{\infty}=-T_{s_1s_2}^{\infty}-T_{s_1}^0T_{s_1s_2}^{\infty}-T_{s_1}^1T_{s_1s_2}^{\infty}+q^{-1}(T_{s_1s_2}^0+T_{s_2}^0)(T_{s_1s_2}^1+T_{s_2}^1)(T_{s_2s_1}^{\infty}+T_{s_1}^{\infty})-q^{-1}(T_{s_2s_1}^0+T_{s_3}^0)(T_{s_2s_1}^1+T_{s_3}^1)
\end{equation}

\noindent To prove Equation \ref{miracle0}, observe that it rearranges to Equation \ref{miracle'}, which we prove in Section \ref{miracleproof}.

\begin{equation}\label{miracle'}
\mathrm{Avg}_1^{01}(T_{s_1s_2}^{\infty}+q^{-1}T_{s_2}^{01}(T_{s_1}^{01}-T_{s_1}^{\infty})-q^{-1}T_{s_2}^{S}T_{s_1}^{\infty})=0
%\mathrm{Avg}_1^{01}(T_{s_1s_2}^{\infty}+q^{-1}T_{s_2s_1}^{01}-q^{-1}T_{s_2}^{01}\mathrm{Avg}_2^{\infty}T_{s_1}^{\infty})=0
\end{equation}

\end{proof}

\begin{lem}
    $\mathrm{Span}_{\mathbb{C}}(M)$ is closed under $T_{s_2}^0$ and $T_{s_2}^1$.
\end{lem}

\begin{proof}
It is sufficient to check closure under $T_{s_2}^0$. Consider a monomial $m=T_{w_0}^0T_{w_1}^1T_{w_{\infty}}^{\infty} \in M$. $T_{s_2}^0 m\in M $ unless $(w_0,w_{\infty})$ is one of the following

\begin{itemize}
\item $(1,s_2),(1,s_1s_2),(1,s_3)$
\item $(s_1s_2,s_1),(s_1s_2,s_2s_1)$
    
\end{itemize}
In each case, it is straighforward calculation to check that $T_{s_2}^0 \in \mathrm{Span}_{\mathbb{C}}(M)$.

\end{proof}

\subsubsection{$\lambda=\rho$}

$\widetilde{C}^{\rho}$ is generated over $(\mathcal{H}^{fin})^{\otimes S}$ by $J_{\rho},J_{\alpha_1},J_{\alpha_2},J_{-\rho}$. We filter $\widetilde{C}^{\rho}$ by subsets of $\{s_1,s_2\}$. $F^{\emptyset}(\widetilde{C}^{\rho})$ is the submodule generated by $J_{\rho}$. For simple reflection $s_i$, $F^{\{s_i\}}$ is the submodule generated by $J_{\rho}$ and $J_{s_i \cdot \rho}$. $F^{\{s_1,s_2\}}=\widetilde{C}^{\rho}$. By Equations \ref{rhoorbit1}, \ref{rhoorbit2}, and \ref{rhoorbit3}, the following are true in the associated graded:

$$\mathrm{Avg}_{1}^sJ_{\alpha_2}=\mathrm{Avg}_2^{s}J_{\alpha_1}=0 \ \text{for} \ s \in S$$

$$\mathrm{Avg}_{i}^sJ_{-\rho}=0 \ \text{for} \ i \in \{1,2\}, \ s \in S$$

\noindent Therefore,

$$\mathrm{dim}_{\mathbb{C}}(\mathrm{Gr}^{\emptyset}(\widetilde{C}^{\rho}) \leq \mathrm{dim}_{\mathbb{C}}((\mathcal{H}^{fin})^{\otimes S})=|W|^3$$

$$\mathrm{dim}_{\mathbb{C}}(\mathrm{Gr}^{\{s_1\}}(\widetilde{C}^{\rho}) \leq \mathrm{dim}_{\mathbb{C}}((\mathcal{H}^{fin})^{\otimes S}/\langle \mathrm{Avg}_1 \rangle_{s \in S})=\mathrm{dim}_{\mathbb{C}}((\mathcal{H}^{fin}/\mathrm{Avg}_1)^{\otimes S})=3^3$$

$$\mathrm{dim}_{\mathbb{C}}(\mathrm{Gr}^{\{s_2\}}(\widetilde{C}^{\rho}) \leq \mathrm{dim}_{\mathbb{C}}((\mathcal{H}^{fin})^{\otimes S}/\langle \mathrm{Avg}_2 \rangle_{s \in S})=\mathrm{dim}_{\mathbb{C}}((\mathcal{H}^{fin}/\mathrm{Avg}_2)^{\otimes S})=3^3$$

$$\mathrm{dim}_{\mathbb{C}}(\mathrm{Gr}^{\{s_1,s_2\}}(\widetilde{C}^{\rho}) \leq \mathrm{dim}_{\mathbb{C}}((\mathcal{H}^{fin})^{\otimes S}/\langle\mathrm{Avg}_1, \mathrm{Avg}_2 \rangle_{s \in S})=1^3$$

$$\mathrm{dim}_{\mathbb{C}}(\widetilde{C}^{\rho})  \leq 6^3+3^3+3^3+1^3$$

\begin{rmk}
For $\delta \subset \{s_1,s_2\}$, pick an additive character $$\psi_{\delta}:N(\mathbb{F}_q)/[N(\mathbb{F}_q),N(\mathbb{F}_q] \cong \oplus_{\{s_1,s_2\}} \mathbb{F}_q \rightarrow \mathbb{C}^{\times},$$

\noindent that is generic in the arguments $\delta$. We can identify the graded component of $\widetilde{C}^{\rho}$ with the Whittaker module for the finite Hecke algebra. $$\mathrm{Gr}^{\delta}(\widetilde{C}^{\rho}) \cong (C^{(N,\psi_{\delta})}[  \mathcal{B}])^{\otimes S}.$$

\noindent The Whittaker module is the the space of $(N(\mathbb{F}_q),\psi_{\delta})$ equivariant functions on the points of the flag variety. It is a finite Hecke module by convolution after identifying $\mathcal{B} \cong G/B$.

\end{rmk}

\subsubsection{$\langle \check{\alpha_i},\lambda \rangle =0$, $\lambda \neq 0$}

$\widetilde{C}^{\lambda}$ is generated over $(\mathcal{H}^{fin})^{\otimes S}$ by $J_{\lambda}$. Using Equation \ref{wall orbit} we can always write any monomial $T_{w_0}^0T_{w_1}^1T_{w_{\infty}}^{\infty}J_{\lambda}$, $w_s \in W$, as a sum of monomials where 
$$\ell(w_{\infty} s_i) < \ell(w_{\infty}) \implies \ell(w_{0} s_i) > \ell(w_{0}),\ \ell(w_{1} s_i) > \ell(w_{1})$$

\noindent Let us count how many triples $(w_0,w_1,w_{\infty})$ satisfy this condition. There are three $w \in W$ such that $\ell(w s_i) < \ell(w)$ and three such that $\ell(w s_i)>\ell(w)$. The set of $s \in S$ such that $\ell(w_s s_i)<\ell(w)$ is exactly one of the following five: $\emptyset, \{0\},\{1\},\{\infty\},\{0,1\}$.

\subsubsection{$\langle \check{\alpha_i},\lambda \rangle =1$, $\lambda \neq \rho$}

$\widetilde{C}^{\lambda}$ is generated over $(\mathcal{H}^{fin})^{\otimes S}$ by $J_{\lambda}$ and $J_{s_i \cdot \lambda}$. Let $F^0$ be the submodule generated by $J_{\lambda}$. By Equation \ref{shiftedwallorbit}, in the quotient $\widetilde{C}^{\lambda}/F^0$, $\mathrm{Avg}_i J_{s_i \cdot \lambda}=0$. Therefore,

$$\mathrm{dim}_{\mathbb{C}}(\widetilde{C}^{\lambda}) \leq \mathrm{dim}_{\mathbb{C}}(F^0)+\mathrm{dim}_{\mathbb{C}}(\widetilde{C}^{\lambda}/F_0)) \leq \mathrm{dim}_{\mathbb{C}}((\mathcal{H}^{fin})^{\otimes S})+\mathrm{dim}_{\mathbb{C}}((\mathcal{H}^{fin}/\langle \mathrm{Avg}_i\rangle)^{\otimes S})=|W|^3+3^3$$

\subsubsection{$\lambda= \in 2 \rho +\Lambda_+$} $\widetilde{C}^{\lambda}$ is generated by $J_{\lambda}$ under $(\mathcal{H}^{fin})^{\otimes S}$, so $\mathrm{dim}_{\mathbb{C}}(\widetilde{C}^{\lambda}) \leq \mathrm{dim}_{\mathbb{C}}((\mathcal{H}^{fin})^{\otimes S})=|W|^3$

%\section{Application: Variety of $k$ flags}

\section{Directions: Functional Equation, Many Points}
%Automorphic gluing

\subsection{Many Points of Tame Ramification}

We state the following natural generalization of Conjecture \ref{mainconj} to $\mathbb{P}^1$ with several points of tame ramification $S \subset \mathbb{P}^1(\mathbb{F}_q)$, $S \neq \emptyset$.
\begin{conj}\label{genconj}
If $\rho $ is integral then $C_{Eis}$ is the $\mathcal{H}^{\otimes S}$ module generated by $\mathrm{Eis}_0$ with the following relations

\begin{enumerate}
    \item (Translation Relation) For any $\lambda \in \Lambda$ and $p,q \in S$, 
    $$(J_{\lambda}^p-J_{\lambda}^q)\mathrm{Eis}_0=0$$
    
    \item (Reflection Relation)
    For any simple reflection, $s_{\alpha} \in W$ and $p,q \in S$
    $$\left(\prod_{s \in S \setminus \{p\}} \mathrm{Avg}_{s_{\alpha}}^s-\prod_{s \in S \setminus \{q\}}  \mathrm{Avg}_{s_{\alpha}}^s\right)\mathrm{Eis}_0=0$$
\end{enumerate}
\end{conj}

When $S$ consists of two points, the quotient of $\mathcal{H}^{\otimes S}$ by the translation and reflection relation is identified with the regular bimodule of $\mathcal{H}$. In this case, Conjecture \ref{genconj} amounts to identifying $C_{Eis}$ with the regular bimodule. This is done in the categorical geometric setting in Section 2.6 of \cite{Nadler_2019}. A similar argument works in the arithmetic function field setting. The author is not aware of any reference but would be grateful to be referred to one.

\subsection{Reflection Relation as a Functional Equation }

We propose that the reflection relation from Conjecture \ref{mainconj} could be related to the functional equation for Eisenstein series. Let $\mathrm{Bun}_T(\mathbb{P}^1,S)$ be the space classifying pairs $(\mathcal{E},\{(V_s,F_s^0,F_s^1,\tau_s)\}_{s \in S})$, where $\mathcal{E}$ is a $T$-bundle on $\mathbb{P}^1$, $V_s$ is a vector space, and $F_s^0,F^s_1 \subset V_s$ are flags with an identification $\tau_s: \mathrm{Gr}(F_s^0) \cong \mathcal{E}|_s$. The constant term space is the space of compactly supported functions on the rational points of $\mathrm{Bun}_T(\mathbb{P}^1,S)$.

$$\mathrm{CT}:=C[\mathrm{Bun}_T(\mathbb{P}^1,S)]$$

\noindent The functional equation for Eisenstein series expresses that parabolic induction $\mathrm{Eis}:\mathrm{CT} \rightarrow C_{Aut}$ intertwines an action of the Weyl group on the constant term space. $\mathrm{CT}$ is identified with the quotient of $\mathcal{H}^{\otimes S}$ by the translation relation. 

\begin{center}

\begin{tikzcd}
     \mathcal{H}^{\otimes S} / \mathrm{translation} \arrow[r,"\cong"] \arrow[d,"\pi"] & \mathrm{CT} \arrow[d,"\mathrm{Eis}"] \\
    \widetilde{C_{\mathrm{Eis}}} \arrow[r,"\overset{?}{\cong}"] & C_{\mathrm{Eis}}
\end{tikzcd}

\end{center}

\noindent The constant term space is free of rank $|W|^{|S|}$ over $\mathbb{C}[\Lambda]$. In light of the functional equation it is natural to conjecture that $C_{\mathrm{Eis}}$ is free of rank $|W|^{|S|-1}$.

\nocite{*}

\bibliographystyle{plain}
\bibliography{refs}

\begin{thebibliography}{1}

\bibitem{goertz}
Ulrich G{\"o}rtz.
\newblock Alcove walks and nearby cycles on affine flag manifolds.
\newblock {\em Journal of Algebraic Combinatorics}, 26(4):415--430, april 2007.

\bibitem{hkp}
Thomas~J. Haines, Robert Kottwitz, and Amritanshu Prasad.
\newblock Iwahori-hecke algebras.
\newblock {\em Journal of the Ramanujan Mathematical Society}, 25(2):113 --
  145, june 2010.

\bibitem{Lusztig}
George Lusztig.
\newblock Affine hecke algebras and their graded versions.
\newblock {\em Journal of the American Mathematical Society}, 2(3):599--635,
  july 1989.

\bibitem{Nadler_2019}
David Nadler and Zhiwei Yun.
\newblock Geometric langlands correspondence for sl(2), pgl(2) over the pair of
  pants.
\newblock {\em Compositio Mathematica}, 155(2):324--371, feb 2019.

\bibitem{ram}
Kendra Nelsen and Arun Ram.
\newblock Kostka-foulkes polynomials and macdonald spherical functions.
\newblock {\em Surveys in Combinatorics, C. Wensley ed., London Math. Soc.
  Lect. Notes}, pages 325--370, 2003.

\end{thebibliography}

\section{Appendix: Proof of Equation \ref{miracle'}}\label{appendix}

We first found Equation \ref{miracle'} and its proof with algebra software. It is helpful to first verify Equation \ref{miracle'} in $C_{Eis}^0$ to see why it could be true in $\widetilde{C}^0$.

\subsection{Geometric Intepretation of Equation \ref{miracle'}}% in $C_{Eis}^0$}
Define the function $f \in C_{Eis}^0$. 

$$f:=T_{s_1s_2}^{\infty}\mathrm{Eis}_0+q^{-1}T_{s_2}^{01}(T_{s_1}^{01}-T_{s_1}^{\infty})\mathrm{Eis}_0-q^{-1}T_{s_2}^{S}T_{s_1}^{\infty}\mathrm{Eis}_0$$

\noindent Consider the projection forgetting the lines $\ell_s$ at $s \in \{0,1\}$:
$$\pi:\mathrm{Bun}_G(\mathbb{P}^1,S) \rightarrow \mathrm{Bun}_G(\mathbb{P}^1,S,\{0,1\},s_1)$$

\noindent $\mathrm{Avg}_1^{01}=\pi^*\pi_!$. Verifying Equation \ref{miracle'} in $C_{\mathrm{Eis}}^0$ is equivalent to checking that $\pi_!f=0$.

$$T_{s_1s_2}^{\infty} \mathrm{Eis}_0=c_0(1,s_1s_2,s_1s_2)$$

$$T_{s_2}^{01}(T_{s_1}^{01}-T_{s_1}^{\infty})\mathrm{Eis}_0=T_{s_2}^{01}c_0(s_1,s_1,s_1)=c_0(s_3,s_1s_2,s_1s_2)$$

$$T_{s_2}^ST_{s_1}^{\infty}=c_0(s_2,s_3,s_3)+c_0(1,s_3,s_3)$$

\noindent The generic locus of $c_0^{01,s_1}(\emptyset) \subset \mathrm{Bun}_G(\mathbb{P}^1,S,\{0,1\},s_1)$ is where the bundle is trivial and all parabolic data are pairwise transverse. Define  $c_0^{01,s_1}(01)$ as the locus where the bundle is trivial and $p_0, p_1$ coincide but are transverse to $(\ell_{\infty},p_{\infty})$. Comparing stabilizers,

$$\pi_!c_0(1,s_1s_2,s_1s_2)=q^{-1}\pi_!c_0(1,s_3,s_3)=c_0^{01,s_1}(01)$$

$$\pi_!c_0(s_3,s_1s_2,s_1s_2)=\pi_!c_0(s_2,s_3,s_3)=c_0^{01,s_1}(\emptyset)$$

\subsection{Proof of Equation \ref{miracle'}}\label{miracleproof}

Returning to the proof of Equation \ref{miracle'}, from the reflection relation

$$\mathrm{Avg}_2^0\mathrm{Avg}_2^1=\mathrm{Avg}_2^{\infty}\mathrm{Avg}_2^1$$

$$ \implies T_{s_2}^0\mathrm{Avg}_2^1=T_{s_2}^{\infty}\mathrm{Avg}_2^1$$

$$ \implies T_{s_2s_1}^{\infty}T_{s_2}^0\mathrm{Avg}_2=T_{s_2s_1}^{\infty}T_{s_2}^{\infty}\mathrm{Avg}_2^1$$

$$ \implies T_{s_2}^ST_{s_1}^{\infty}=T_{s_3}^{\infty}+(T_{s_2}^1T_{s_3}^{\infty}-T_{s_2}^0T_{s_2s_1}^{\infty})$$

\noindent Therefore,

$$\mathrm{Avg}_1^{01}(T_{s_1s_2}^{\infty}+q^{-1}T_{s_2}^{01}(T_{s_1}^{01}-T_{s_1}^{\infty})-q^{-1}T_{s_2}^{S}T_{s_1}^{\infty})=\mathrm{Avg}_1^{01}  (T_{s_1s_2}^{\infty}-q^{-1}T_{s_3}^{\infty})+q^{-1}\mathrm{Avg}_1^{01}\left(T_{s_2}^{01}(T_{s_1}^{01}-T_{s_1}^{\infty})-(T_{s_2}^1T_{s_3}^{\infty}-T_{s_2}^0T_{s_2s_1}^{\infty})\right)$$

\noindent Equation \ref{miracle'} follows from the following two lemmas.

\begin{lem}
    $$\mathrm{Avg}_1^{01}T_{s_1s_2}^{\infty}=q^{-1}\mathrm{Avg}_1^{01}T_{s_3}^{\infty}$$
\end{lem}
\begin{proof}
    From the reflection relation

$$\mathrm{Avg}_1^{\infty}\mathrm{Avg}_1^1=\mathrm{Avg}_1^0\mathrm{Avg}_1^1$$

$$\implies T_{s_1}^{\infty} \mathrm{Avg}_1^1=T_{s_1}^0\mathrm{Avg}_1^1$$

$$\implies q^{-1}\mathrm{Avg}_1^{01}T_{s_3}^{\infty}=q^{-1}T_{s_1s_2}^{\infty}\mathrm{Avg}_1^0T_{s_1}^{\infty}\mathrm{Avg}_1^1=q^{-1}T_{s_1s_2}^{\infty}\mathrm{Avg}_1^0T_{s_1}^0\mathrm{Avg}_1^1=q^{-1}T_{s_1s_2}^{\infty}(q\mathrm{Avg}_1^0)\mathrm{Avg}_1^1=\mathrm{Avg}_1^{01}T_{s_1s_w}^{\infty}$$

\end{proof}

\begin{lem} $$\mathrm{Avg}_1^{01}T_{s_2}^{01}(T_{s_1}^{01}-T_{s_1}^{\infty})=\mathrm{Avg}_1^{01}(T_{s_2}^1T_{s_3}^{\infty}-T_{s_2}^0T_{s_2s_1}^{\infty})$$

\end{lem}
\begin{proof} First, observe that by the reflection relation $T_{s_2}^1T_{s_3}^{\infty}-T_{s_2}^0T_{s_2s_1}^{\infty} =T_{s_2s_1}^{\infty}(T_{s_2}^{01}-T_{s_2}^{\infty})$. Then, check that

$$q\mathrm{Avg}_1^{01}T_{s_2}^{01}(T_{s_1}^{01}-T_{s_1}^{\infty})-q\mathrm{Avg}_1^{01}T_{s_2s_1}^{\infty}(T_{s_2}^{01}-T_{s_2}^{\infty})= $$ $$A \cdot \mathrm{Avg}_1^0(\mathrm{Avg}_{1}^1-\mathrm{Avg}_{1}^{\infty})+B \cdot \mathrm{Avg}_1^1(\mathrm{Avg}_{1}^0-\mathrm{Avg}_{1}^{\infty})+C \cdot \mathrm{Avg}_2^0(\mathrm{Avg}_{2}^1-\mathrm{Avg}_{2}^{\infty})+D \cdot \mathrm{Avg}_2^1(\mathrm{Avg}_{2}^0-\mathrm{Avg}_{2}^{\infty})=0,$$

\noindent where

$$A=qT_{s_2}^1+T_{s_2}^1T_{s_1s_2}^{\infty}+T_{s_2}^1T_{s_3}^{\infty}+(q-1)T_{s_1s_2}^1-T_{s_1s_2}^1T_{s_1}^{\infty}-T_{s_1s_2}^1T_{s_2}^{\infty}-T_{s_1s_2}^1T_{s_2s_1}^{\infty}-qT_{s_2}^0-qT_{s_2}^{0\infty}-T_{s_2}^0T_{s_3}^{\infty}+qT_{s_2}^{01}+T_{s_2}^{01}T_{s_1s_2}^{\infty} $$ $$+(q-1)T_{s_2}^0T_{s_1s_2}^1-T_{s_2}^{0\infty}T_{s_1s_2}^1-qT_{s_1s_2}^0-qT_{s_1s_2}^0T_{s_2}^{\infty}-T_{s_1s_2}^0T_{s_3}^{\infty}+qT_{s_1s_2}^0T_{s_2}^1+T_{s_1s_2}^{0\infty}T_{s_2}^1+(q-1)T_{s_1s_2}^{01}-T_{s_1s_2}^{01}T_{s_2}^{\infty}$$

$$B=(q-1)T_{s_2}^0T_{s_1s_2}^{\infty}-T_{s_2}^{01}T_{s_1s_2}^{\infty}+T_{s_2}^0T_{s_1s_2}^1+T_{s_2}^{0\infty}T_{s_1s_2}^1+(q-1)T_{s_1s_2}^{0\infty}-T_{s_1s_2}^{\infty}T_{s_2}^1+T_{s_1s_2}^{01}+T_{s_1s_2}^{01}T_{s_2}^{\infty}$$

$$C=qT_{s_1}^{\infty}-qT_{s_1}^1-qT_{s_1}^1T_{s_2s_1}^{\infty}-(q-1)T_{s_2s_1}^1+T_{s_2s_1}^1T_{s_1}^{\infty}+T_{s_2s_1}^1T_{s_2}^{\infty}+T_{s_2s_1}^{1\infty}+qT_{s_1}^{0\infty}-qT_{s_1}^{01}-qT_{s_1}^{01}T_{s_2s_1}^{\infty}-(q-1)T_{s_1}^0T_{s_2s_1}^1+T_{s_1}^{0\infty}T_{s_2s_1}^1$$ $$+T_{s_1}^0T_{s_2s_1}^1T_{s_2}^{\infty}+T_{s_1}^0T_{s_2s_1}^{1\infty}+qT_{s_2s_1}^0T_{s_1}^{\infty}+T_{s_2s_1}^0T_{s_1s_2}^{\infty}+T_{s_2s_1}^0T_{s_3}^{\infty}-T_{s_2s_1}^0T_{s_1}^1+(q-1)T_{s_2s_1}^0T_{s_1}^{1\infty}-T_{s_2s_1}^0T_{s_1}^1T_{s_2}^{\infty}-T_{s_2s_1}^{0\infty}T_{s_1}^1$$

$$D=-qT_{s_1}^{\infty}-qT_{s_2s_1}^{\infty}+qT_{s_1}^1-qT_{s_1}^{0\infty}-qT_{s_1}^0T_{s_2s_1}^{\infty}+qT_{s_1}^{01}$$

\end{proof}

\end{document}